\documentclass[12pt,oneside,reqno]{amsart}

\usepackage[T1]{fontenc}
\usepackage[utf8]{inputenc}
\usepackage{lmodern}                 
\usepackage[margin=1in]{geometry}
\usepackage{microtype}
\usepackage{graphicx}
\usepackage{xcolor}
\usepackage{booktabs}                
\usepackage{enumitem}                
\usepackage{verbatim}                
\usepackage{csquotes}                
\usepackage{tikz}

\usepackage{amsmath,amssymb,amsthm}
\usepackage{aliascnt}            
\usepackage{mathtools}               
\usepackage{mathrsfs}                
\usepackage{bm}                      

\numberwithin{equation}{section}     

\definecolor{linknavy}{RGB}{20,55,120}
\definecolor{citeolive}{RGB}{20,90,60}
\definecolor{urlplum}{RGB}{110,30,70}

\usepackage{hyperref}
\hypersetup{
  colorlinks   = true,
  linkcolor    = linknavy,           
  citecolor    = citeolive,
  urlcolor     = urlplum,
  linktoc      = all,
  breaklinks   = true,
  bookmarksnumbered = true,
  pdfstartview = FitH
}

\usepackage[
  backend   = biber,
  style     = numeric-comp,          
  sorting   = nyt,                   
  maxnames  = 99,                    
  minnames  = 99,                    
  giveninits= true,                  
  isbn      = false,
  url       = false,
  doi       = true,
  eprint    = true,
  backref   = false
]{biblatex}
\defbibheading{bibliography}{\section*{\refname}}
\newcommand*{\namedash}{\textup{--}}          

\DeclareNameFormat{familydash}{%
  \namepartfamily
  \ifnumcomp{\value{listcount}}{<}{\value{liststop}}{\namedash}{}%
}

\DeclareCiteCommand{\citeauthors}[\mkbibparens]
  {\usebibmacro{prenote}}
  {\printnames[familydash][1-99]{labelname}}
  {\multicitedelim}
  {\usebibmacro{postnote}}

\DeclareCiteCommand*{\citeauthors}
  {\usebibmacro{prenote}}
  {\printnames[familydash][1-99]{labelname}}
  {\multicitedelim}
  {\usebibmacro{postnote}}

\DeclareCiteCommand{\citeauthorsyear}[\mkbibparens]
  {\usebibmacro{prenote}}
  {\printnames[familydash][1-99]{labelname}%
   \namedash
   \iffieldundef{year}{\printfield{labelyear}}{\printfield{year}}}
  {\multicitedelim}
  {\usebibmacro{postnote}}

\DeclareCiteCommand*{\citeauthorsyear}
  {\usebibmacro{prenote}}
  {\printnames[familydash][1-99]{labelname}%
   \namedash
   \iffieldundef{year}{\printfield{labelyear}}{\printfield{year}}}
  {\multicitedelim}
  {\usebibmacro{postnote}}

\usepackage[capitalise,noabbrev,nameinlink]{cleveref}

\theoremstyle{plain}
\newtheorem{theorem}{Theorem}[section]
\newaliascnt{lemma}{theorem}
\newtheorem{lemma}[lemma]{Lemma}
\aliascntresetthe{lemma}
\newaliascnt{proposition}{theorem}
\newtheorem{proposition}[proposition]{Proposition}
\aliascntresetthe{proposition}
\newaliascnt{corollary}{theorem}
\newtheorem{corollary}[corollary]{Corollary}
\aliascntresetthe{corollary}
\newaliascnt{conjecture}{theorem}

\aliascntresetthe{conjecture}

\theoremstyle{definition}
\newaliascnt{definition}{theorem}

\aliascntresetthe{definition}
\newaliascnt{example}{theorem}

\aliascntresetthe{example}
\newaliascnt{remark}{theorem}
\newtheorem{remark}[remark]{Remark}
\aliascntresetthe{remark}
\newaliascnt{notation}{theorem}

\aliascntresetthe{notation}
\newaliascnt{convention}{theorem}

\aliascntresetthe{convention}
\newaliascnt{assumption}{theorem}

\aliascntresetthe{assumption}
\newaliascnt{claim}{theorem}

\aliascntresetthe{claim}
\newaliascnt{question}{theorem}

\aliascntresetthe{question}
\newaliascnt{problem}{theorem}

\aliascntresetthe{problem}
\newaliascnt{openproblem}{theorem}

\aliascntresetthe{openproblem}

\theoremstyle{plain}
\newtheorem*{theorem*}{Theorem}
\newtheorem*{lemma*}{Lemma}
\newtheorem*{proposition*}{Proposition}
\newtheorem*{corollary*}{Corollary}
\theoremstyle{definition}
\newtheorem*{definition*}{Definition}
\newtheorem*{example*}{Example}
\newtheorem*{remark*}{Remark}
\newtheorem*{notation*}{Notation}
\newtheorem*{convention*}{Convention}
\newtheorem*{claim*}{Claim}

\theoremstyle{plain}
\newtheorem{knowntheorem}{Theorem}
\newaliascnt{knownlemma}{knowntheorem}
\newtheorem{knownlemma}[knownlemma]{Lemma}
\aliascntresetthe{knownlemma}
\newaliascnt{knownproposition}{knowntheorem}
\newtheorem{knownproposition}[knownproposition]{Proposition}
\aliascntresetthe{knownproposition}
\newaliascnt{knowncorollary}{knowntheorem}
\newtheorem{knowncorollary}[knowncorollary]{Corollary}
\aliascntresetthe{knowncorollary}
\theoremstyle{definition}
\newaliascnt{knowndefinition}{knowntheorem}

\aliascntresetthe{knowndefinition}

\crefname{theorem}{Theorem}{Theorems}
\Crefname{theorem}{Theorem}{Theorems}
\crefname{lemma}{Lemma}{Lemmas}
\Crefname{lemma}{Lemma}{Lemmas}
\crefname{proposition}{Proposition}{Propositions}
\Crefname{proposition}{Proposition}{Propositions}
\crefname{corollary}{Corollary}{Corollaries}
\Crefname{corollary}{Corollary}{Corollaries}
\crefname{conjecture}{Conjecture}{Conjectures}
\Crefname{conjecture}{Conjecture}{Conjectures}
\crefname{definition}{Definition}{Definitions}
\Crefname{definition}{Definition}{Definitions}
\crefname{example}{Example}{Examples}
\Crefname{example}{Example}{Examples}
\crefname{remark}{Remark}{Remarks}
\Crefname{remark}{Remark}{Remarks}
\crefname{notation}{Notation}{Notations}
\Crefname{notation}{Notation}{Notations}
\crefname{convention}{Convention}{Conventions}
\Crefname{convention}{Convention}{Conventions}
\crefname{assumption}{Assumption}{Assumptions}
\Crefname{assumption}{Assumption}{Assumptions}
\crefname{claim}{Claim}{Claims}
\Crefname{claim}{Claim}{Claims}
\crefname{question}{Question}{Questions}
\Crefname{question}{Question}{Questions}
\crefname{problem}{Problem}{Problems}
\Crefname{problem}{Problem}{Problems}
\crefname{openproblem}{Open Problem}{Open Problems}
\Crefname{openproblem}{Open Problem}{Open Problems}

\crefname{knowntheorem}{Theorem}{Theorems}
\Crefname{knowntheorem}{Theorem}{Theorems}
\crefname{knownlemma}{Lemma}{Lemmas}
\Crefname{knownlemma}{Lemma}{Lemmas}
\crefname{knownproposition}{Proposition}{Propositions}
\Crefname{knownproposition}{Proposition}{Propositions}
\crefname{knowncorollary}{Corollary}{Corollaries}
\Crefname{knowncorollary}{Corollary}{Corollaries}
\crefname{knowndefinition}{Definition}{Definitions}
\Crefname{knowndefinition}{Definition}{Definitions}

\crefname{equation}{equation}{equations}
\Crefname{equation}{Equation}{Equations}
\crefname{section}{Section}{Sections}
\Crefname{section}{Section}{Sections}
\crefname{subsection}{Subsection}{Subsections}
\Crefname{subsection}{Subsection}{Subsections}
\crefname{figure}{Figure}{Figures}
\Crefname{figure}{Figure}{Figures}
\crefname{table}{Table}{Tables}
\Crefname{table}{Table}{Tables}
\crefname{cases}{Case}{Cases}
\Crefname{cases}{Case}{Cases}
\crefname{subcases}{Subcase}{Subcases}
\Crefname{subcases}{Subcase}{Subcases}
\crefname{enumi}{item}{items}
\Crefname{enumi}{Item}{Items}

\newenvironment{proofof}[1]
  {\begin{proof}[Proof of \cref{#1}]}
  {\end{proof}}

\newcounter{cases}
\newcounter{subcases}

\newcommand{\R}{\mathbb{R}}

\newcommand{\Z}{\mathbb{Z}}

\newcommand{\T}{\mathbb{T}}
\newcommand{\M}{\mathcal{M}}
\newcommand{\hM}{\widehat{\mathcal{M}}}
\newcommand{\cH}{\mathcal{H}}
\newcommand{\Sp}{\mathfrak{S}}                 
\newcommand{\wg}{\R^{n}\times\T^{m}}           

\newcommand{\eps}{\varepsilon}

\DeclarePairedDelimiter{\abs}{\lvert}{\rvert}
\DeclarePairedDelimiter{\norm}{\lVert}{\rVert}

\DeclarePairedDelimiterX{\ip}[2]{\langle}{\rangle}{#1,#2}
\DeclareMathOperator{\Tr}{Tr}

\newcommand{\Com}{\mathrm{Com}}

\allowdisplaybreaks                            
\setlist[enumerate]{itemsep=2pt,topsep=4pt}
\setlist[itemize]{itemsep=2pt,topsep=4pt}
\numberwithin{figure}{section}
\numberwithin{table}{section}

\begin{document}

\title[Pointwise convergence for Schr\"odinger equations]{Pointwise convergence for  single-particle   Schr\"odinger equation and  Schr\"odinger equation with infinitely many particles   on the torus and waveguide manifolds}

\author{Divyang G. Bhimani}
\address{Department of Mathematics, Indian Institute of Science
  Education and Research Pune, Homi Bhabha Road, Pune 411008, India}
\email{divyang.bhimani@iiserpune.ac.in}

\author{Subhash R. Choudhary}
\address{Department of Mathematics, Indian Institute of Science
  Education and Research Pune, Homi Bhabha Road, Pune 411008, India}
\email{subhashranjan.choudhary@students.iiserpune.ac.in}


\begin{abstract}
	In 1980, Carleson posed a question about the least regularity required for initial data in a Sobolev space $H^s$ to ensure pointwise convergence of the solution to the single particle linear Schr\"odinger equation. In this paper, we study Carleson problem on the $d$-dimensional  waveguide manifold $\R^n\times\T^m$ and initiate its analogy for the system of infinitely many orthonormal particles (arising from the transition of  many-body quantum mechanics to the thermodynamic limit)  on torus and waveguide manifold.
	
	For a single particle on the waveguide manifold, we establish a maximal-in-time Strichartz estimate that yields almost everywhere convergence to the initial data in $H^s$ for $s > \frac{d}{d+2}$. On the other hand, we adopt Bourgain's counterexample approach to show that this convergence fails for $s<  \frac{d}{2(d+1)}.$ For infinitely many particles, we generalize the classical maximal-in-time Strichartz estimate and  pointwise convergence  result  of  \citeauthors*{CompaanLuca2021pointwise}  \cite{CompaanLuca2021pointwise} in the setting of fermionic systems on the torus. Similar result is also established for the waveguide manifold.   We also establish necessary condition (in the spirit of Bourgain's counterexample) for the pointwise convergence problem for fermionic systems.
	These are the first results in the setting of torus and waveguide manifold, and complement the works of  \citeauthors*{BezLeeNakamura2020maximalose} \cite{BezLeeNakamura2020maximalose} and \citeauthors*{BezKinoshitaShinyaShiraki2024} \cite{BezKinoshitaShinyaShiraki2024} on Euclidean space.
\end{abstract}
\maketitle
\tableofcontents
\section{Introduction}\label{sec:intro}
\subsection{Single particle Schr\"odinger equation}
In this paper, we first  study Carleson's problem regarding pointwise convergence for the linear Schr\"odinger equation on waveguide manifold $\mathbb R^n \times \mathbb T^m.$ Indeed, the Carleson problem \cite{Carleson1980point}   goes back to 1980s, which asks  for the minimal regularity  of the initial  data  $f$ so that time-dependent  solution of free  Schr\"odinger equation $u(x,t)$   converges to $f(x)$ pointwise at  time $t\to 0$. Since then this problem has a long history (to be described briefly in \Cref{pwones} below).  Specifically, we begin with the Schr\"odinger equation of the following form:
\begin{equation}\label{NLS}
	\begin{cases}
		i\partial_tu +\Delta u = 0\\
		u(t=0)=f
	\end{cases} (t,z) \in \mathbb R \times \mathcal{M}.  
\end{equation}
Here $\mathcal{M}= \mathbb R^d$ (Euclidean space) or $\mathbb T^d$ (torus) or $\mathbb R^n \times \mathbb T^m$ (waveguide manifold). We recall that 
\begin{align}\label{about manifold}
	\widehat{\mathcal{M}}= \begin{cases}
		\mathbb R^d & if \quad \mathcal{M}=\mathbb R^d\\
		\mathbb Z^d & if \quad \mathcal{M}=\mathbb T^d\\
		\mathbb R^n \times \mathbb Z^m & if \quad   \mathcal{M}=\mathbb R^n \times \mathbb T^m  
	\end{cases},   
\end{align}
where $\widehat{\mathcal{M}}$ denotes the Pontryagin dual of $\mathcal{M}.$ Here 
$e^{it \Delta_\mathcal{M}}f(z)$ denotes the  Schr\"odinger propagator:
\[
\mathcal{F}\left[e^{it \Delta_\mathcal{M}}f(z)\right](\xi) 
= \begin{cases}
	e^{2 \pi it \abs{\xi}^2 } \mathcal{F}f (\xi) & if \quad \xi \in \mathbb{R}^d \text{ or } \mathbb{Z}^d \\
	e^{2 \pi it (\abs{\xi_{1}}^2 +\abs{{\xi_{2}}}^{2}) } \mathcal{F}f (\xi) & if \quad \xi=(\xi_{1}, \xi_{2}) \in \mathbb{R}^n \times \mathbb{Z}^m
\end{cases} , 
\]
where $\mathcal{F}$ is the Fourier transform defined on $\M.$ Then the solution of the \cref{NLS} is $u(t,z)=e^{it \Delta_{\M}}f(z).$ 
As we mentioned earlier, the problem of identifying the smallest exponent $s>0$ for which 
\begin{equation}\label{e:carleson}
	\lim_{t\to0} e^{it \Delta_{\M}}f(z) = f(z) \quad {\rm a.e.}\;\;\; z\in \M
\end{equation}
holds for all $f\in H^s(\M)$ originated in the famous paper by Carleson \cite{Carleson1980point} for $\M=\R.$ Here, 
$$H^s(\mathcal{M}) = \langle \nabla_{\M} \rangle^{-s}L^2(\M)=\left\{ f \in \mathcal{D}'(\mathcal{M}) :  \left(1+|\xi|^2 \right)^{\frac{s}{2}} \mathcal{F}(f) \in L^2(\hM) \right\},$$
is the inhomogeneous Sobolev space of order $s$.
The standard way to tackle  this pointwise convergence problem is to consider maximal-in-time estimates of the form 
\begin{equation*}
	\big\|\sup_{t \in \R} |e^{it \Delta_{\M}}f| \big\|_{L^q_z(\M)} \le C\|f\|_{H^s(\M)}
\end{equation*}
or its local variants,  since standard arguments allow one to deduce \eqref{e:carleson} for all $f \in H^s(\M)$. Whilst local space-time bounds of this type suffice for the purpose of deducing \eqref{e:carleson}.

We denote $P_Nf$ the  Littlewood-Paley projection operators (see \eqref{lpope}). Our first main result gives maximal-in-time Strichartz estimates on waveguide manifold. Specifically we have the following theorem.
\begin{theorem}\label{max stri for wave}
	For every $N\geq 1,\varepsilon > 0,$
	\begin{equation*}
		\| \sup_{0 < t \leq 1} \left| e^{it\Delta} P_N f(x,y) \right| \|_{L^2_{x,y} \left(B^n(0,1)\times \mathbb{T}^m \right)}
		\lesssim_\varepsilon N^{\frac{n+m}{n+m+2}+\varepsilon}
		\|f\|_{L^2(\mathbb{R}^n \times \mathbb{T}^m)}
	\end{equation*}
	holds whenever $f \in L^2(\mathbb{R}^n \times \mathbb{T}^m).$
\end{theorem}
Consequently, we have the following pointwise convergence result   on waveguide manifold.
\begin{corollary}[sufficient condition]\label{genind}
	For every $f \in H^s(\R^n \times \T^m)$ with $s> \frac{n+m}{n+m+2},$ we have 
	\begin{equation*}
		\lim_{t \to 0} e^{it\Delta} f(x,y)=f(x,y)
	\end{equation*}
	holds for almost every $(x,y) \in \R^n \times \T^m.$
\end{corollary}
We refer to  Remark \ref{pst1.1} for the comments on the proof. 
We do not know whether  the regularity index $s$ in \Cref{genind} is optimal or not. However, we could  establish necessary condition for the point-wise convergence  in the following theorem.
\begin{theorem}[necessary condition]\label{Necessary condition for theorem}
	Let $n,m\geq1$. For $s < \frac{n+m}{2(n+m+1)},$  there exists $f \in H^s(\mathbb{R}^n \times \mathbb{T}^m)$ such that
	\[
	\lim_{t \to 0} e^{it\Delta} f(x,y) \neq f(x,y) \quad { for \ \rm a.e.}\;\;\; (x,y) \in  \mathbb{R}^n \times \mathbb{T}^m.
	\]
\end{theorem}
We point out that well-posedness theory for  nonlinear Schr\"odinger equation on waveguide manifold has been studied by many authors. See for e.g. Barron \cite{BarronAlex}, Takaoka--Tzvetkov \cite{Takaoka2001} and Deng--Fan--Yang--Zhao--Zheng \cite{Deng2024JFA}. However, to the best of author's knowledge \Cref{max stri for wave,Necessary condition for theorem} are the first results that deal with the pointwise convergence problem. This also complements  the extensive literature available on torus and Euclidean space (see \Cref{pwones}).
\begin{remark}  The point-wise convergence  problem for $s \in  \left[ \frac{n+m}{2 (n+m+1)},  \frac{n+m}{ (n+m+2)}  \right] $ remains an interesting open question.   Indeed, the regularity threshold for pointwise convergence obtained here also coincides with the expected sharp threshold suggested by the Euclidean and torus literature. See \Cref{pwones}. In the waveguide setting \(\mathbb{R}^n\times\mathbb{T}^m\), the same exponent appears with \(d=n+m\), reflecting the mixed Euclidean-periodic geometry of the problem. Hence, our results show that the sharp behaviour predicted separately by the Euclidean and torus theories persists in the waveguide setting as well. Cf. \Cref{Necessary condition for torus}. 
\end{remark}

\begin{remark}[notations]  We simply write $\nabla$ instead of $\nabla_{\M}$ and $\Delta$ instead of $\Delta_{\M}$. It will be clear from the context, without any ambiguity, which manifold we are working on.  We write \( A \lesssim B \) to denote that there exists a constant \( C > 0 \) such that \( A \leq C B \). Similarly, we use \( A \lesssim_u B \) to indicate that there exists a constant \( C(u) > 0 \), depending on \( u \in \R \), such that \( A \leq C(u) B \).
	We write \( A \sim B \) if both \( A \lesssim B \) and \( B \lesssim A \) hold. And we also define the following
	\begin{equation}\label{A(N) defined}
		A_{\M}(N)
		:= \left\{ \xi \in \hM : |\xi| \sim N \right\}, \qquad N \geq 1.
	\end{equation}
\end{remark}

\subsubsection{\textbf{Prior works:}}\label{pwones} We briefly mention  development  on Carleson's problem  on $\mathbb R^d$ and $\mathbb T^d$.\\
$-$ \textbf{Pointwise convergence  on $\M=\R^d$}.
In this case when $d = 1,$ Carleson \cite{Carleson1980point} proved the convergence of \eqref{e:carleson} for $s \ge \frac{1}{4}$, whereas it fails for $s < \frac{1}{4}$ in any dimension, as shown by Dahlberg--Kenig \cite{DahlbergKenig1982}.
In higher dimensions, Sj\"olin \cite{Sjolin1987regularity} and Vega \cite{Vega1988pointwise} independently showed that \eqref{e:carleson} holds for $s > \frac{1}{2}.$ This result was improved to $s > \frac{1}{2} - \frac{1}{4d}$ by Lee \cite{Lee2006pointwiseR2} for $d = 2$ and by Bourgain \cite{Bourgain2013higherdimension} for $d \ge 3.$ Subsequently, Bourgain \cite{Bourgain2016note} showed that $s \ge \frac{d}{2(d + 1)}$ is necessary for the almost everywhere convergence. The sufficiency part of the convergence was shown by Du--Guth--Li \cite{DuGuthLi2017sharpR2} when $d = 2$ and by Du--Zhang \cite{DuZhang2019maximalL2} when $d \ge 3$ for a sharp range except for the endpoint. See also \cite{Bourgain1995essay,Carbery1985radial,Cowling1983pointwise,DuGuthLi2018pointwisemultilinear,LucaRogerkeith2017fractal,LucaRogerKeith2019pointwise,MoyuaVergasVega1999restriction} for previous work.

The following proposition was first proved in the case $n=1$ with the
$L^4$-norm by Kenig--Ponce--Vega \cite{Kenig_1991_OsciDis}; in this case,
the choice of the $L^2$-norm is sub-optimal. Later, Du--Guth--Li
\cite{DuGuthLi2017sharpR2} proved the result for $n=2$ using polynomial
partitioning and $\ell^2$-decoupling estimates. Finally, Du--Zhang
\cite{DuZhang2019maximalL2} extended the result to all dimensions
$n\geq 3$.
\begin{knownproposition}[Theorem 2.2 in \cite{DuZhang2019maximalL2}]\label{maximal stri for Euclidean}
	Let $n \geq 1.$ For all $\varepsilon > 0$ and all $N \geq 1$, the inequality
	\begin{equation}\label{e:maximal stri for eucli}
		\left\| \sup_{0 \le t \le N} |e^{it\Delta} f| \right\|_{L^2(B^n(0,N))}
		\lesssim_\varepsilon N^{\frac{n}{2(n+1)}+\varepsilon} \|f\|_{L^2(\mathbb{R}^n)}
	\end{equation}
	holds whenever $f \in L^2(\mathbb{R}^n)$ satisfies $\operatorname{supp} \widehat{f} \subset B^n(0,1).$
\end{knownproposition}
$-$ \textbf{Pointwise convergence  on $\M=\T^d$:}
Compared to the Euclidean case, the Carleson problem \eqref{e:carleson} on the torus is more subtle, due to the erratic behavior of the periodic Schr\"odinger kernel, which is linked to analytic number theory. Moyua--Vega \cite{MoyuaVega2008boundsfor} proved a maximal-in-time Strichartz estimate on $\T$, at frequency scale $N$ with exponent greater than $\frac13$ and deduced almost everywhere convergence for general initial data on $\mathbb{T}$ under the condition $s > \frac{1}{3}$. Following Moyua--Vega's argument together with the Strichartz inequality on $\mathbb{T}^d$, thanks to the decoupling theory, Compaan--Luc\`a--Staffilani \cite{CompaanLuca2021pointwise} proved the convergence for $s > \frac{d}{d+2}$ (see green region in \Cref{fig:s_range}), which is the best known result up to now and recovers the Moyua--Vega exponent at $d=1$.
\begin{knowntheorem}[Compaan--Luc\`a--Staffilani \cite{CompaanLuca2021pointwise}]\label{max stri for torus}
	Let $d \geq 1.$ For every $N\geq 1,\varepsilon > 0,$
	\begin{equation*}
		\| \sup_{0 < t \leq 1} \left| e^{it\Delta} P_Nf(x) \right| \|_{L^2_{x} \left(\mathbb{T}^d \right)}
		\lesssim_\varepsilon N^{\frac{d}{d+2}+\varepsilon}
		\|f\|_{L^2(\mathbb{T}^d)}
	\end{equation*}
	holds whenever $f \in L^2(\mathbb{T}^d).$ 
\end{knowntheorem}

As a consequence of \Cref{max stri for torus}, they proved the following. 
\begin{knowncorollary}[\cite{CompaanLuca2021pointwise}]\label{ptonT}
	Let $d \geq 1.$ Then for every $f \in H^s(\T^d)$ with $s> \frac{d}{d+2},$
	\begin{equation*}
		\lim_{t \to 0} e^{it\Delta} f(x)=f(x)
	\end{equation*}
	holds for almost every $x \in \T^d.$
\end{knowncorollary}
\begin{remark}\label{Necessary condition for torus}
	\begin{enumerate}
		\item If 
		$s<\frac{d}{2(d+1)}$ (see in red region in \Cref{fig:s_range}),
		then pointwise convergence fails, while the endpoint case
		$s=\frac{d}{2(d+1)}$ remains open. For the proof of this see \cite[Theorem 1.1]{EceizabarrenaRenato2022fractalsperiodic}.
		\item  For special initial data $f$ such that
		\[
		f(x) := \sum_{|k|\sim N} e^{ik\cdot x}, \qquad N \in \mathbb{N},
		\]
		the Schr\"odinger flow can be regarded as a type of Weyl sum. Barron \cite{Barron2022L4maximal} investigated the $L^4$ maximal estimate for the Weyl sum on $\mathbb{T}$ based on the Weyl-sum estimates and the circle method. For the higher dimensional torus $\mathbb{T}^d$, $d \ge 2$, Miao--Yuan--Zhao \cite{MiaoYuanZhao2023maximalestimate} proved the convergence for this special initial data for $s > \frac{d}{2(d+1)}$ (see blue region in \Cref{fig:s_range}) which is sharp up to the endpoint. 
	\end{enumerate}
\end{remark}

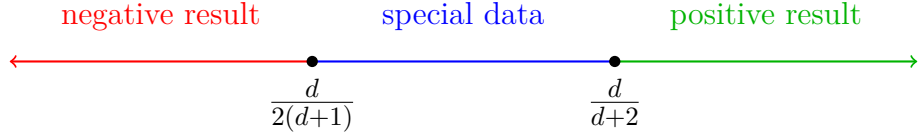
\begin{figure}[h]
	\centering
	
	\begin{tikzpicture}[
		every node/.style={font=\large},
		dot/.style={circle, fill=black, inner sep=1.5pt}
		]
		\draw[<-, red, thick] (0,0) -- (4,0);
		\node[red, font=\normalsize] at (2,0.6) {negative result};
		\draw[-, blue, thick] (4,0) -- (8,0);
		\node[blue, font=\normalsize] at (6,0.6) {special data};
		\draw[->, green!70!black, thick] (8,0) -- (12,0);
		\node[green!70!black, font=\normalsize] at (10,0.6) {positive result};
		
		\node[dot, label=below:{$\frac{d}{2(d+1)}$}] at (4,0) {};
		\node[dot, label=below:{$\frac{d}{d+2}$}] at (8,0) {};
	\end{tikzpicture}
	\caption{The range of $s$ is divided into three regions by these values 
		$\frac{d}{2(d+1)}$ and $\frac{d}{d+2}$. The red segment is where convergence fails (\Cref{Necessary condition for torus}). The blue segment including $\frac{d}{2(d+1)}$ is where convergence is known and sharp only for the special Weyl-sum data of Miao--Yuan--Zhao. The green segment is where convergence holds for general $H^s$ data (\Cref{ptonT}).}
	\label{fig:s_range}
	
\end{figure}

\subsection{Schr\"odinger equation with infinitely many particles} In recent years, there has been ongoing interest to study Schr\"odinger equation with infinitely many particles, i.e. infinite system of \eqref{NLS} with infinite sequence of orthonormal initial data. Specifically, we wish to consider infinite system of \eqref{NLS}, that is:
\begin{equation}\label{InSNLS}
	\begin{cases}
		i\partial_tu_j +\Delta u_j = 0\\
		u_j(t=0)=f_j
	\end{cases}\quad  j \in \mathbb N,  
\end{equation}
where $(f_j)$ are orthonormal in $H^s(\M)$. 
Indeed, it is well-known that system  \eqref{InSNLS} and its nonlinear analogue   models infinitely many fermions. See \cite{chen2017global,chen2018scattering, sabin2014hartree,FrankJEMS2014,BenjaminCMP2014}.  In order to motivate our main results, we first briefly describe background.  Consider $M$ coupled of Hartree equations that describe the dynamics of $M$ fermions interacting via a potential $\omega: \mathcal{M} \to \mathbb R$ as follows:
\begin{align}\label{HN}
	\left\{\begin{array}{rlrl}
		i \partial_t u_1 & =\left(-\Delta+\omega * \rho\right) u_1, & & \left.u_1\right|_{t=0}=f_1, \\
		& \vdots \\
		i \partial_t u_M & =\left(-\Delta+\omega * \rho\right) u_M, & & \left.u_M\right|_{t=0}=f_M,
	\end{array}\right.
\end{align}
where $(t,z) \in \mathbb{R} \times \mathcal{M},$ and to account for the Pauli principle, the family $\left\{f_j\right\}_{j=1}^M$ is assumed to be an orthonormal system in $L^2\left(\mathcal{M} \right)$. Here, $\rho$ is the total density function of particles defined by 
$$\rho(t, z)=\sum_{j=1}^M\left|u_j(t, z)\right|^2.$$ 
Let us introduce the density matrix corresponding to \eqref{HN} as
\begin{eqnarray*}
	\gamma_M (t) = \sum_{j=1}^M | u_j(t)\rangle \langle u_j(t) |,
\end{eqnarray*}
where $|u \rangle \langle v|$ denotes the operator $f\mapsto \langle v, f\rangle u$ from $L^2\left(\mathcal{M} \right)$ into itself. Then system \eqref{HN} can be rewritten as an operator-valued equation
\begin{equation}\label{oHN}
	\begin{cases}
		i \partial_t\gamma_{M} = [-\Delta+ \omega \ast \rho_ {\gamma_M}, \gamma_M]\\
		\gamma_M(t=0) = \sum_{j=1}^M | f_{j}\rangle \langle f_{j} |
	\end{cases}.
\end{equation}
Here, $[A, B]=AB-BA$ is the operator commutator
and the density function $\rho_ {\gamma_M}$ is given by 
\[\rho_{\gamma_M} (t,z)= \gamma_{M} (t, z,z), \] 
where $\gamma_M (t, \cdot, \cdot)$ is the integral kernel of operator $\gamma_M(t)$.
We are going to consider the case when $M \rightarrow \infty$. In this case, we naturally arrive at the following operator-valued equivalent formulation 
\begin{align}\label{Hartree omega} 
	\left\{\begin{array}{l}
		i \partial_t \gamma=\left[-\Delta+\omega * \rho_\gamma, \gamma\right] \\
		\left.\gamma\right|_{t=0}=\gamma_0
	\end{array}.\right.
\end{align}
Here, the unknown $\gamma=\gamma(t)$ is a bounded self-adjoint operator on $L^2(\mathcal{M})$ representing the one-particle density matrix and $\rho_\gamma:\mathbb R \times \mathcal{M} \to \mathbb R$ is the scalar \textbf{density function} associated to $\gamma(t),$ formally defined by
\begin{equation}\label{dfcp}
	\rho_\gamma(t, z)=\gamma (t, z,z),
\end{equation}
where $\gamma (t, z, z')$ denotes (with the abuse of notation) the integral kernel of the operator $\gamma(t),$ i.e.,
\[
\gamma (t) \phi (z) = \int_{\mathcal{M}} \gamma (t, z, z') \phi (z') dz'.
\]  
It is known that  \cref{Hartree omega} describes the mean-field dynamics of an interacting gas containing infinitely many fermions on $\mathcal{M}.$ In the Euclidean space, this setting was studied by Frank--Sabin \cite{frank2017restriction}, who developed orthonormal Strichartz estimates and applied them to the well-posedness of the Hartree equation in the Schatten  space. See also \cite{FengMathZ}.  This was followed by the work of Chen--Hong--Pavlović \cite{chen2017global,chen2018scattering} on the global well-posedness and scattering of infinitely many fermions. See also further development by Hadama et al. in  \cite{HadamaJFA,HadamaJMP}. On the torus, the corresponding problem was studied by Nakamura \cite{nakamura2020orthonormal}, who established orthonormal Strichartz estimates and applied them to the local well-posedness of the periodic Hartree equation for infinitely many particles. While for the partial regularity refinement of this, see \cite{bhimani2026refined}. More recently, on the waveguide manifold $\mathbb{R}^n\times\mathbb{T}^m$, this problem was studied by Bhimani--Choudhary \cite{bhimani2025orthonormal}.

We are now ready to come back to  view our original problem for fermionic system, that is, 
our goal is to investigate the convergence of the solution of system \eqref{Hartree omega} when $\omega=0$ i.e., 
\begin{align}\label{Hartree} 
	\left\{\begin{array}{l}
		i \partial_t \gamma=\left[-\Delta, \gamma\right] \\
		\left.\gamma\right|_{t=0}=\gamma_0
	\end{array}.\right.
\end{align}
We may regard \eqref{Hartree} as equivalent to \eqref{InSNLS}, see e.g. \cite{sabin2014hartree}.  Since this is an operator-valued equation, the natural framework to study this problem is to consider initial data $\gamma_0$ from the Schatten space; and in order to make sense of point wise convergence, we consider corresponding density functions (see \eqref{dfcp} and below \eqref{e:carleson}). 

For $s\in\mathbb{R}$ and $1\le \alpha\le\infty$, we work with the Sobolev-Schatten space $\Sp^{\alpha,s}(L^2(\M))$, which is obtained by introducing Sobolev regularity into the classical Schatten spaces. More precisely,
\[
\Sp^{\alpha,s}(L^2(\M))
:=
\left\{
A:   A \ \text{is compact and } \|\langle\nabla\rangle^sA\langle\nabla\rangle^s\|_{\Sp^\alpha (L^2(\M))}<\infty
\right\},
\]
where $\Sp^\alpha$ denotes the classical Schatten space (\Cref{schp}). Here  $\langle \nabla \rangle^s$ is the inhomogeneous derivative and $\langle \nabla \rangle^s \phi = \left((1 + |\eta|^2)^{\frac{s}{2}} \hat{\phi}\right)^\vee$. Whenever the underlying Hilbert space $\cH$ is clear from the context, we simply write $\Sp^{\alpha,s}$ instead of $\Sp^{\alpha,s}(\cH)$. We note that Schatten spaces enjoy monotonicity property, specifically, we have
\[
\Sp^{\beta_1,s}\subseteq\Sp^{\beta_2,s},
\qquad
\beta_1\le\beta_2.
\]
For the precise definition and basic properties of Sobolev-Schatten spaces, we refer the reader to \Cref{schp,sec7}. Taking these considerations into account, we now formulate the \textit{convergence problem for infinite system  \eqref{Hartree}} as follows:
\begin{itemize}
	\item[--] what is a   possibly large class  of initial data $\gamma_0 \in \Sp^{\alpha}(\cH)$ (i.e. large $\alpha$) and the smallest value of $s>0$ so that  
	pointwise convergence of the associated density functions  (of the solution $\gamma(t)$ of \eqref{Hartree})
	\begin{equation}\label{e:carleson-density}
		\lim_{t\to0} \rho_{\gamma(t)}(z) = \rho_{\gamma_0}(z)\quad { for \ \rm a.e.}\;\;\; z\in\M
	\end{equation} 
	holds true.
\end{itemize}

In order to understand convergence problem \eqref{e:carleson-density} on torus $\mathbb T^d$, we first establish following theorem. 
\begin{theorem}\label{t:maximal ose T}
	Assume that $s>\frac{d}{4}+\frac{d}{2(d+2)}, d \geq1$ and   $1 \leq \alpha'\le 2$.
	Then  the following maximal-in-time bound
	\begin{equation}\label{e:max ose T }
		\bigg\| \sum_{j=1}^{\infty} \lambda_j \, |e^{it\Delta} f_j|^2 \bigg\|_{L_x^2 L_t^\infty(\mathbb T^{d+1})}
		\lesssim \, \|\lambda\|_{\ell^{\alpha'}}
	\end{equation}
	holds true
	for any orthonormal system $(f_j)_j$ in $H^s(\mathbb T^d)$ and any sequence of coefficients
	$\lambda=(\lambda_j)_j\in \ell^{\alpha'}$.
\end{theorem}

\Cref{t:maximal ose T} should be regarded as an orthonormal extension of classical result of  Compaan, Luc\`a, and Staffilani (\Cref{max stri for torus}). Indeed, by taking $\lambda_1=1$ and $\lambda_j=0$ for $j>1$, we can recover \Cref{max stri for torus}. However, we need to pay price of taking larger $s$ in \Cref{t:maximal ose T} as compare to regularity index $s$ in \Cref{max stri for torus}. As a consequence of \Cref{t:maximal ose T}, we have the following. 

\begin{corollary}[pointwise convergence on $\mathbb T^d$]\label{c:pointwise T}
	Assume that $\frac{d}{4}+\frac{d}{2(d+2)}<s<\frac{d}{2}, d \geq1$ and $\gamma_0 \in \Sp^{\alpha',s}$ is self-adjoint with $\alpha' < 2$. Then the densities
	$\rho_{\gamma(t)}(x)$ and $\rho_{\gamma_0}(x)$ are well defined, and they satisfy
	\eqref{e:carleson-density}.
\end{corollary}

Next, we shall turn our attention to convergence problem \eqref{e:carleson-density} on waveguide manifold $\mathbb R^n \times \mathbb T^m.$

\begin{theorem}\label{t:maximal ose wave}
	Let $n,m\geq1$ and $s>\frac{n+m}{4}+\frac{n+m}{2(n+m+2)}$. Then we have the maximal-in-time bound
	\begin{equation}\label{e:max ose wave}
		\bigg\| \sum_j \lambda_j \, |e^{it\Delta} f_j|^2 \bigg\|_{L_x^2 L_t^\infty(B^n(0,1) \times \T^m \times [0,1])}
		\le C \, \|\lambda\|_{\ell^{\alpha'}}
	\end{equation}
	for any orthonormal system $(f_j)_j$ in $H^s(\R^n \times \T^m)$ and any sequence of coefficients
	$\lambda=(\lambda_j)_j\in \ell^{\alpha'}$, provided that $1 \le \alpha'\le 2$.
\end{theorem}
As a consequence of \Cref{t:maximal ose wave}, we have the following.
\begin{corollary}[pointwise convergence on $\mathbb R^n \times \mathbb T^m$]\label{c:pointwise wave}
	Assume that $\frac{n+m}{4}+\frac{n+m}{2(n+m+2)}<s<\frac{n+m}{2}$ and that $\gamma_0 \in \Sp^{\alpha',s}$ is self-adjoint with $\alpha' < 2$. Then the densities
	$\rho_{\gamma(t)}(x,y)$ and $\rho_{\gamma_0}(x,y)$ are well defined, and they satisfy
	\eqref{e:carleson-density}.
\end{corollary}
 \Cref{c:pointwise wave} gives sufficient condition for the point wise convergence. On the other hand,  next  \Cref{necessary for infinite} establish necessary condition for the pointwise convergence. 
\begin{corollary}[necessary condition for infinite equations]\label{necessary for infinite}
	Let $\gamma_0 \in \Sp^{\alpha',s}(L^2(\M))$ is self-adjoint with $\alpha' \geq 1$ and $\mathcal{M}$ be as in \eqref{about manifold}. Suppose that the densities
	$\rho_{\gamma(t)}(z)$ and $\rho_{\gamma_0}(z)$ are well defined, and they satisfy
	\eqref{e:carleson-density}. Then $s  \geq \frac{d}{2(d+1)},$ where $d$ is the dimension of $\M.$	
\end{corollary}
\begin{remark}
	\Cref{t:maximal ose wave} and \Cref{c:pointwise wave} are the orthonormal extension of \Cref{max stri for wave} and \Cref{genind} respectively.  Indeed, by taking $\lambda_1=1$ and $\lambda_j=0$ for $j>1$ in \Cref{t:maximal ose wave} recovers \Cref{max stri for wave}. As on the torus, this comes at the cost of a larger regularity threshold, from $s>\frac{d}{d+2}$ in \Cref{genind} to $\frac{d}{4}+\frac{d}{2(d+2)}<s<\frac{d}{2}$ in \Cref{c:pointwise wave}, a loss of at least $\frac{d^2}{4(d+2)}$ at the lower end, together with a new upper cutoff at $s=\frac d2$ with no counterpart in \Cref{genind}, $d=n+m$.
\end{remark}

To the best of the author's knowledge, \Cref{t:maximal ose T,t:maximal ose wave} and \Cref{c:pointwise T,c:pointwise wave} are new in the literature.  On the other hand the pointwise convergence problem  \eqref{e:carleson} on  $\mathbb{R}^d$ is well studied in the literature and it is due to Bez, Kinoshita and Shiraki \cite{BezKinoshitaShinyaShiraki2024}.

\subsubsection{\textbf{Prior works}} We briefly mention the  ongoing development of \eqref{Hartree omega} and \eqref{Hartree}. The pointwise convergence for the infinite system \eqref{Hartree} was first studied by Bez, Lee and Nakamura \cite{BezLeeNakamura2020maximalose} in the one-dimensional Euclidean space. Specifically, they established the following.
\begin{knowntheorem} [Bez--Lee--Nakamura, \cite{BezLeeNakamura2020maximalose}]
	The weak-type maximal-in-time estimate
	\begin{equation}\label{1.15}
		\left\|
		\sum_j \lambda_j \left| e^{it\Delta} f_j \right|^2
		\right\|_{L_x^{2,\infty}L_t^\infty(\mathbb{R}\times\mathbb{R})}
		\lesssim
		\|\lambda\|_{\ell^{\alpha'}}
	\end{equation}
	holds for orthonormal functions $(f_j)_j$ in $\dot{H}^{\frac14}(\mathbb{R})$ and $\alpha'<2$. Consequently, if $\gamma_0 \in \Sp^{\alpha',\frac14}$
	and $\alpha'<2$, then the density function $\rho_{\gamma(t)}$ satisfies \eqref{e:carleson-density} for almost every $x\in\mathbb{R}$.
\end{knowntheorem}
Further development in this direction was done by Bez, Kinoshita and Shiraki in \cite{BezKinoshitaShinyaShiraki2024} for the fractional Laplacian operator. We mention their particular case in the next theorem.  Let $\beta\in(0,d]$ be any real number, the Borel measure $\mu$ on $\mathbb{R}^d$ is said to be $\beta$-dimensional if
\[
\sup_{x\in\mathbb{R}^d,\,r>0}
\frac{\mu(B(x,r))}{r^\beta}
<\infty.
\]
We shall use $\mathfrak{M}^\beta(\mathbb{B}^d)$ to denote the collection of all $\beta$-dimensional probability measures supported on the unit ball $\mathbb{B}^d$. 
\begin{knowntheorem}[Bez-Kinoshita-Shiraki, \cite{BezKinoshitaShinyaShiraki2024}]\label{2.2}
	Let $d\geq 1$, $s\in[0,\frac d2)$, $\beta\in(0,d]$, $\alpha'\geq 1$,  $\mu\in\mathfrak{M}^\beta(\mathbb{B}^d)$, and  $s\geq \frac d4$. Then the maximal estimate 
	\begin{equation}\label{2.5}
		\left\|
		\sum_j \lambda_j \left| e^{it\Delta} f_j \right|^2
		\right\|_{L_x^1(\mathbb{B}^d,d\mu)L_t^\infty(0,1)}
		\lesssim
		\|\lambda\|_{\ell^{\alpha'}},
	\end{equation}
	holds for all orthonormal systems $(f_j)_j$ in $\dot{H}^s(\mathbb{R}^d),$ whenever
	\[
	s>\frac12\left(d-\frac{\beta}{\alpha'}\right).
	\]
	This is sharp in the sense that if
	\[
	s<\frac12\left(d-\frac{\beta}{\alpha'}\right),
	\]
	then there exist $\mu\in\mathfrak{M}^\beta(\mathbb{B}^d)$, an orthonormal system $(f_j)_j$ in $\dot{H}^s(\mathbb{R}^d)$, and a sequence $(\lambda_j)_j\in \ell^{\alpha'}$ such that \eqref{2.5} fails.
\end{knowntheorem}

We introduce the divergence set
\[
\mathcal{D}(\gamma_0)
:=
\left\{
x\in\mathbb{R}^d :
\lim_{t\to 0}\rho_{\gamma(t)}(x)\neq \rho_{\gamma_0}(x)
\right\}
\]
associated with the initial data $\gamma_0$, and where $\gamma(t)$ is the solution of \eqref{Hartree}.
\begin{knowncorollary}[\cite{BezKinoshitaShinyaShiraki2024}]\label{2.3}
	Let $d\geq 1$, $s\in [0,\frac d2)$. Suppose $\gamma_0\in \Sp^{\alpha',s}$ is self-adjoint with $1\leq \alpha' < \frac{d}{d-2s}.$
	For $s\geq \frac d4$, the pointwise convergence
	\[
	\lim_{t\to 0}\rho_{\gamma(t)}(x)=\rho_{\gamma_0}(x)
	\]
	holds and furthermore we have
	\[
	\dim_H \mathcal{D}(\gamma_0)\leq (d-2s)\alpha',
	\]
	where $\dim_H \mathcal{D}(\gamma_0)$ is Hausdorff dimension of $\mathcal{D}(\gamma_0).$
\end{knowncorollary}
\subsection{Further comments}

\begin{remark} Our main results for fermionic system \eqref{Hartree} deserve several comments.
	\begin{enumerate}
		\item  Pointwise convergence results (\Cref{c:pointwise T,c:pointwise wave}) are the periodic and waveguide counterparts of \Cref{2.3}.
		\item \Cref{necessary for infinite} plays the role of a converse to this, in the same spirit as the sharpness statement in \Cref{2.2}.
		\item The weak-type maximal estimate and  the weak-type norm $L_x^{2,\infty}$ \eqref{1.15} is restricted to special dimension $d=1$. While \Cref{c:pointwise T,c:pointwise wave} hold for every $d=n+m\ge1$, indeed, the density convergence directly follows from a strong-type $L^2_x$ maximal estimate.
		\item \Cref{2.2} covers general dimension  $d$ but in a weighted norm $L_x^1(\mathbb B^d,d\mu)$ over $\beta$-dimensional measures $\mu\in\mathfrak M^\beta(\mathbb B^d)$, with threshold $s>\frac12(d-\beta/\alpha')$ depending on $\beta$. While \Cref{t:maximal ose T,t:maximal ose wave} give an unweighted $L^2_x$ estimate on $\T^d$ and on $\R^n\times\T^m$, at the fixed threshold $s>\frac{d}{4}+\frac{d}{2(d+2)}$, reflecting the same periodic Strichartz exponent $\frac{d}{d+2}$ (\Cref{max stri for torus}) that already distinguishes the torus from the Euclidean case in the single-particle problem. 
		
		\item  The  necessary condition (\Cref{necessary for infinite})  for the pointwise convergence is proved by 
		packing the single-particle non-convergent function furnished by \Cref{Necessary condition for theorem} into a rank-one operator.  And then recovering pointwise convergence of $e^{it\Delta}F$ itself from convergence of $|e^{it\Delta}F|^2$ via polarization against a fixed nowhere-vanishing Schwartz function. In contrast,  \Cref{2.2} records the sharpness of the maximal estimate \eqref{2.5} itself, for the same family of $\beta$-dimensional measures. See also \Cref{Necessary condition for torus}.
	\end{enumerate}
\end{remark}

\begin{remark} We comment on the hypothesis of our main results.
	\begin{enumerate}
		\item  For the proof strategy of \Cref{t:maximal ose T,t:maximal ose wave}, we refer to \Cref{psfs}.  It remains unknown for both theorems whether the condition $\alpha'=2$ is optimal in the range
		\[
		s>\frac{d}{4}+\frac{d}{2(d+2)}.
		\]
		\item In \Cref{t:maximal ose T,t:maximal ose wave} the regularity threshold $s>\frac{d}{4}+\frac{d}{2(d+2)}$ is exactly what makes the dyadic sum $\sum_N N^{\frac12(\sigma-2s+\varepsilon)}$, $\sigma=\frac d2+\frac{d}{d+2}$, converge in the proof. See \eqref{}. 
		\item \Cref{c:pointwise T,c:pointwise wave} require $s>\frac{d}{4}+\frac{d}{2(d+2)},$ because of \Cref{t:maximal ose T,t:maximal ose wave}. While the upper bound $s<\frac d2$ comes from \Cref{lem:lieb-sobolev-product}, it plays role only in defining $\rho_{\gamma_0}$ and $\rho_{\gamma(t)}$. And it is not known whether either the exponent $s$ or the range $\alpha'<2$ is optimal.
		\item In \Cref{c:pointwise T,c:pointwise wave}, the condition $\alpha'<2$ arises from the inclusion (see the footnote in \Cref{sec7})
		\[
		\ell^{\alpha'} \subseteq \ell^{2,1}, \qquad \alpha'<2.
		\]
	\end{enumerate}
\end{remark}

The rest of the paper is organized as follows. \Cref{pre} introduces our notation, function spaces, and Schatten-space background. \Cref{imp-lem} establishes the tools needed for the main results, including the Lieb--Sobolev inequality on $\R^n \times \T^m$ (\Cref{lem:lieb-sobolev-product}). \Cref{sec5} proves \Cref{max stri for wave} and \Cref{genind}. \Cref{sec-necc-single} proves \Cref{Necessary condition for theorem}. 
\Cref{sec-suff-inf} proves \Cref{t:maximal ose T,t:maximal ose wave}, and \Cref{sec7} proves \Cref{c:pointwise
T,c:pointwise wave}. \Cref{sec-necc-inf} proves \Cref{necessary for infinite}.

\section{Preliminaries}\label{pre}
\textbf{Notations}.
Denote $\langle \xi \rangle^{s}=(1+|\xi|^2)^\frac{s}{2}, s\in \mathbb R.$ Denote $\mathbf{1}_A$ for the characteristic function of a set $A$.  
The notation \( A \approx B \), though less rigorous, will sometimes be used to indicate that \( A \) and \( B \) are essentially of the same size, up to small and insignificant error terms.
We write $B^n(x,R)$ for the ball centred at $x \in \R^n$ of radius $R>0.$
We use usual Lebesgue spaces $L_{z}^p:=L^{p}_{z}{(\M)}$ and $L_{z}^{p}L_{t}^{q}(\M \times I)=L_{z}^{p}(\M,L_{t}^{q}(I)),$ with the following norm
$$\|f\|_{L_{z}^{p}L_{t}^{q}(\M \times I)}=\|\|f\|_{L^q(I)} \|_{L^p(\M)},$$
where $I$ is a bounded subset of $\mathbb{R}$; if not, we will mention it. The \textbf{Lorentz spaces} on a $\sigma$-finite measure space \((Y,\mu)\),  denoted by \(L^{p,q}(Y)\) \((1 \leq p,q < \infty)\), are given by the following norm
\begin{equation}\label{Lorentz Lpq space definition}
	\norm{f}_{L^{p,q}} := \left( p\int_0^\infty \lambda^{q-1} \big(\mu(\{x\in Y : \abs{f(x)} > \lambda\})\big)^{q/p}\, d\lambda \right)^{1/q}
	<\infty.
\end{equation}
For $q= \infty,$ the \textbf{weak \(L^p\) norm} \((1 \leq p < \infty)\) is defined by
\begin{equation}\label{weak Lp space definition}
	\|f\|_{L^{p,\infty}(Y)}
	:=
	\sup_{\lambda>0}
	\lambda\,\mu\big(\{x\in Y:|f(x)|>\lambda\}\big)^{1/p}
	<\infty.
\end{equation}
Throughout the rest of the paper, we fix the notation
\begin{equation}\label{XxI}
	X \times I=\begin{cases}
		\T^d \times \T \quad &if \quad \M=\T^d \\
		B^n(0,1) \times \T^m \times (0,1] \quad &if \quad  \M=\R^n \times \T^m
	\end{cases}.
\end{equation}
The \textbf{Lorentz sequence space} $\ell^{p,q}(\mathbb{N}) \ \  ( 1\leq p<\infty, 1\leq q\leq \infty)$ is defined by the norm
\begin{equation}\label{Lorentz sequence space}
	\|\lambda\|_{\ell^{p,q}}
	:=
	\left\|
	j^{\frac1p-\frac1q}\lambda_j^*
	\right\|_{\ell^q(\mathbb{N})},
\end{equation}
where $(\lambda_j^*)_j$ denotes the non-increasing rearrangement of $(|\lambda_j|)_j$. For $1\le r\le\infty$ we write $r'$ for the conjugate exponent, $\tfrac1r+\tfrac1{r'}=1$,
with $1'=\infty$ and $\infty'=1$; in particular $\alpha'$, $p'$ and $q'$ always denote
the conjugates of $\alpha$, $p$ and $q$.
Let us denote \begin{equation*}
	z=\begin{cases}
		x \in \mathbb{R}^d & if \quad z \in \mathbb{R}^d \\
		x \in \mathbb{T}^d & if \quad z \in \mathbb{T}^d \\
		(x,y) \in \mathbb{R}^n \times \mathbb{T}^m & if \quad z \in \mathbb{R}^n \times \mathbb{T}^m
	\end{cases},
\end{equation*}
and if $\xi \in \mathbb{R}^n \times \mathbb{Z}^m$ then $\xi=(\xi_{1},\xi_{2}),$ where $\xi_{1} \in \mathbb{R}^n$ and $\xi_{2} \in \mathbb{Z}^m.$ We define the Fourier transform $\mathcal{F}$ on $\mathcal{M}$ (see \eqref{about manifold}) and inverse Fourier transform $\mathcal{F}^{-1}$ as follows:
$$(\mathcal{F}f)(\xi)=\widehat{f}(\xi)=\int_{\mathcal{M}} f(z) e^{-2 \pi i z \cdot \xi}dz, \quad \xi \in \widehat{\mathcal{M}}, $$
$$(\mathcal{F}^{-1}f)(z)={f}^{\vee}(z)=\int_{\hM} f(\xi) e^{2 \pi i z \cdot \xi} \,d\xi, \quad z \in \mathcal{M}.$$
The \textbf{Sobolev spaces} $H^s(\M)$ and $\dot{H}^s(\M)$  on $\M$ are given by the following norms
$$\|f\|_{ H^{s}(\M)}= \|\langle \nabla \rangle^{s} f\|_{L^2(\M)} \quad \text{ and } \quad \|f\|_{ \dot{H}^{s}(\M)}= \|\lvert \nabla \rvert^{s} f\|_{L^2(\M)}$$
respectively, where  
$$\mathcal{F}[\langle \nabla_{\M} \rangle^{s} f] (\xi)
=\langle \xi \rangle^{s} \, \mathcal{F}(f)(\xi) \quad \text{and} \quad \mathcal{F}[|\nabla_{\M}|^{s} f] (\xi)
=|\xi|^{s} \, \mathcal{F}(f)(\xi), \qquad \xi \in \hM.
$$
\textbf{Partition of unity.} Let $\eta_{1}: \mathbb{R} \to [0,1]$ be a smooth even function satisfying 
\begin{equation*}
	\eta_{1}(x)= \begin{cases}
		1 &  if \quad \abs{x} \leq 1 \\
		0  &  if \quad \abs{x} \geq 2
	\end{cases}.
\end{equation*}
Then we form $\eta$ on $\mathbb R^{n+m}= \mathbb R^d$ as follows:
\[\eta(\xi)=\eta_{1}(\xi'_{1})\eta_{1}(\xi'_{2}) \cdots \eta_{1}(\xi'_{n+m}), \quad \xi=(\xi'_{1}, \cdots, \xi'_{n+m}) \in \mathbb{R}^n \times \mathbb{R}^m.\]
For dyadic integer $N \in 2^{\mathbb{N}},$ we define \textbf{Littlewood-Paley projectors} on $\M$ by 
\begin{equation}\label{P less N}
	\widehat{P_{1}f}(\xi)= \eta(\xi) \widehat{f}(\xi), \quad   \widehat{P_{\leq N}f}(\xi)= \eta\left(\frac{\xi}{N}\right) \widehat{f}(\xi), \quad \xi \in \hM,
\end{equation}
and 
\begin{equation}\label{lpope}
	P_N f = P_{\le N}f-P_{\le N/2}f,  \quad \text{where} \quad  P_{\leq \frac{1}{2}}=0.
\end{equation}
Next, define $\psi(x)=\eta(x)-\eta(2x)$ and
\begin{equation}\label{psi N}
	\psi_N(x)=
	\begin{cases}
		\eta(x) & N=1\\
		\psi(N^{-1}x) & N>1
	\end{cases},
\end{equation}
then $\operatorname{supp}\psi_{N}=\{x: |x| \sim N\}$ for all $N>1.$
Then, using \eqref{lpope}, we obtain
\[
\widehat{P_N f}(\xi)=\psi_N(\xi)\,\widehat f(\xi).
\]
Moreover, by \eqref{psi N}, we get 
$$ \sum_{N \in 2^{\mathbb{N}}} \psi_{N}(\xi)=1, \qquad \xi \in \mathbb{R}^d.$$ 
We recall \textbf{Bernstein inequality}: For every $1\le p\le q\le\infty$, $s\ge0$, and dyadic $N\ge1$,
\begin{equation}\label{Bernstein}
	\||\nabla|^sP_Nf\|_{L^q(\M)} \lesssim N^{s+d\left(\frac{1}{p}-\frac{1}{q}\right)} \|P_Nf\|_{L^p(\M)},
\end{equation}
where $d$ is the dimension of the manifold $\M.$

\subsection{Basic properties of Schatten spaces $\Sp^{\alpha}$} \label{schp}
Let $\mathcal{B}(\cH)$ denote the space of all bounded operators on a Hilbert space $\cH.$   
For $\alpha \in [1, \infty),$ define   Schatten $\alpha$-norm of $A \in \mathcal{B}(\cH)$ by 
$$\|A\|_{\Sp^{\alpha}}=\|A\|_{\Sp^{\alpha}(\cH)} = (\operatorname{Tr}|A|^{\alpha})^{\frac{1}{\alpha}}.$$ 
The \textbf{Schatten space} $\Sp^{\alpha }(\cH) \ (\alpha \in [1, \infty))$ consists of all compact operators $A \in \mathcal{B}(\cH)$ such that $\operatorname{Tr}|A|^{\alpha} < \infty,$ where $|A| = \sqrt{A^* A}, A^* \text{ is adjoint of }A$.  
The  Schatten space norm is  given by 
$$\|A\|_{\Sp^{\alpha}}=\|A\|_{\Sp^{\alpha}(\cH)} = (\operatorname{Tr}|A|^{\alpha})^{\frac{1}{\alpha}}.$$ 
The Schatten space norm coincides with the operator norm for  $\alpha = \infty,$ 
thus we  define $$\|A\|_{\Sp^{\infty}}=\|A\|_{\Sp^{\infty}(\cH)} = \|A\|_{\cH\to \cH}.$$ 
Let $\mathcal{G}$ be another Hilbert space. For $1\leq p,q < \infty,$ the \textbf{Lorentz Schatten space $\Sp_{p,q}(\cH,\mathcal{G})$} consists of all compact operators $T:\cH\to\mathcal{G}$ such that the sequence of singular values $s(T) = (s_n(T))_{n\ge 1}$ (in decreasing
order) lies in Lorentz sequence space $\ell^{p,q}$ (see \eqref{Lorentz sequence space}), with norm $$\norm{T}_{\mathfrak S_{p,q}} = \norm{s(T)}_{\ell^{p,q}}.$$
When $p=q$ and $\cH=\mathcal{G},$ then  $\Sp_{p,p}(\cH,\mathcal{G})=\Sp^{p}(\cH)$ (Schatten space).
For $q\geq 1$, the \textbf{weak Lorentz Schatten space $\Sp_{q,\infty}(\cH,\mathcal{G})$} consists of compact operators $T:\cH\to\mathcal{G}$ for which
\[
\|T\|_{\Sp_{q,\infty}}^{q} = \sup_{\kappa>0}\, \kappa^{q}\, n_\kappa(T) <\infty,
\]
where, writing $s_1(T)\geq s_2(T)\geq\cdots$ for the singular values of $T$ (i.e.\ the eigenvalues of $(T^*T)^{1/2}$, listed with multiplicity), $n_\kappa(T)$ denotes the number of singular values that exceed the threshold $\kappa$. For further details see \cite{simon2005trace} and \cite{Frank2014Cwikel}.

Whenever the underlying Hilbert spaces $\cH$ and $\mathcal{G}$ are clear from the context, we simply write $\Sp_{p,q}$ instead of $\Sp_{p,q}(\cH,\mathcal{G})$.

\begin{knownlemma}[Duality principle, Lemma 3 in \cite{frank2017restriction}]\label{PL1} Let $p, q \geq 1,$ and $\alpha \geq 1.$ Let $T$ be a bounded operator from a separable Hilbert space $\cH$ to $L_{z}^{2p}L_{t}^{2q}(\M \times I)$. Then the following are equivalent.
	\begin{enumerate}
		\item[(i)] There is a constant $C > 0$ such that
		\begin{equation*}
			\left\|WTT^*W\right\|_{\Sp^\alpha (\cH)} \leq C\left\|W\right\|^{2}_{L_z^{2p'}L_t^{2q'}(\M \times I)},
		\end{equation*} for all $ W \in {L_z^{2p'}L_t^{2q'}(\M \times I)}.$ Where the function $W$ act as a multiplication operator. 
		\item[(ii)] There is a constant $C' > 0$ such that for any ONS $(f_j)_{j \in \mathbb{Z}} \subset L_{z}^{2}(\mathcal{M})$ and any sequence $(\lambda_j)_{j \in \mathbb{Z}} \subset \mathbb{C}$,
		\begin{equation*}
			\left\| \sum_{j \in \mathbb{Z}} \lambda_j |Tf_j|^2 \right\|_{L_z^{p}L_t^{q}(\M \times I)} \leq C' \left( \sum_{j \in \mathbb{Z}} |\lambda_j|^{\alpha'} \right)^{1/\alpha'}.
		\end{equation*}
\end{enumerate} \end{knownlemma}
\subsection{Known results}
In the next lemma we note  the RHS is independent of phase  $\Psi$, and so we may say it is a phase removal lemma. This  will play a crucial role in the proof of \Cref{Necessary condition for proposition}.
\begin{knownlemma}[Lemma 2.2 in \cite{Pierce2020BourgainCounter}]\label{lem:vdc}
	Let $\mu\in L^1([a,b])$ and let $\Psi \in C^1([a,b];\R)$. Then 
	\[
	\int_a^b\mu(\lambda)\,e^{2\pi i\Psi(\lambda)}\,d\lambda=e^{2\pi i\Psi(b)}\int_a^b\mu(\lambda)\,d\lambda+E. 
	\]
	Moreover, the  constant $E$  satisfies the following estimate 
	\[\abs E\le 2\pi\,\norm{\mu}_{L^1[a,b]}\,\norm{\Psi'}_{L^\infty[a,b]}\,(b-a).\]
\end{knownlemma}
The following lemma compares the Lebesgue measure of $X_Q$ in terms of the number of rational points involved and the sizes $h_1(Q)$ and $h_2(Q)$ of the corresponding boxes.
\begin{knownlemma}[Lemma A.1 in \cite{Eceizabarrena2025BourgainCounter}]\label{lem:A.1}
	Let $Q\ge Q_0$ for some absolute constant $Q_0$ depending only on $n$. Define
	\[
	X_Q=
	\bigcup_{\substack{q\sim Q\\ q\ \mathrm{odd}}}
	\;
	\bigcup_{\substack{(p_1,p')\in[0,q)^n\\ \gcd(p_1,q)=1}}
	B^1\!\left(\frac{p_1}{q},\,h_1(Q)\right)
	\times
	B^{n-1}\!\left(\frac{p'}{q},\,h_2(Q)\right),
	\]
	such that $h_1(Q)\sim 1/Q^\alpha$ and $h_2(Q)\sim 1/Q^\beta$ for some
	$\alpha,\beta\ge 1$. Then,
	\[
	|X_Q|
	\sim
	\frac{Q^{n+1}h_1(Q)h_2(Q)^{\,n-1}}
	{1+Q^{n+1}h_1(Q)h_2(Q)^{\,n-1}},
	\]
	where $|X_Q|$ is the Lebesgue measure of the set $X_Q.$
\end{knownlemma}
Next lemma provides a bound for the supremum of $F$ in terms of its initial value and its $L^p$-norms, with the parameter $\mu$ allowing us to balance the contributions of $F$ and $F'.$
\begin{knownlemma} [Lemma 2.5 in \cite{Lee2006pointwiseR2}]\label{calculus lemma}
	Let \(F\) be a smooth function on the interval \([a,b]\) and let
	\(1\le p\le \infty\). Then for any \(\mu>0\),
	\[
	\sup_{t\in [a,b]} |F(t)|
	\le
	C\left(
	|F(a)|
	+
	\mu^{\frac1p-1}\|F'\|_{L^p([a,b])}
	+
	\mu^{\frac1p}\|F\|_{L^p([a,b])}
	\right),
	\]
	where \(C\) depends only on \(p\).
\end{knownlemma}
The next proposition is known as the  diagonal Strichartz estimate, which is a consequence of the famous $\ell^2$-decoupling inequality for the waveguide manifold, as studied by Barron \cite{BarronAlex}.
\begin{knownproposition}[Proposition 3.4 in \cite{BarronAlex} and Theorem 1.10 in \cite{bhimani2025orthonormal}]\label{diag strichartz}
	Let $I \subset \R$ be any bounded interval and
	\[
	p=\frac{2(n+m+2)}{n+m}.
	\]
	 Then for any $ \eps >0,$ one has 
	\[
	\left\|
	e^{it\Delta}
	P_{\le N}f
	\right\|_{L^p(\mathbb R^n\times\mathbb T^m\times I)}
	\lesssim_{\varepsilon}
	N^{\varepsilon}
	\|f\|_{L^2(\mathbb R^n\times\mathbb T^m)}.
	\]
\end{knownproposition}

\section{Important Lemmas}\label{imp-lem}
\subsection{Maximal estimates implies pointwise convergence}

As an application of Littlewood-Paley theory, we establish the following lemma. This states that if one can control the Schr\"odinger maximal function at each frequency scale with a loss of about $N^{s_0+\varepsilon}$ for $\varepsilon>0$ (see \eqref{AR stri}), then one automatically obtains a global bound (\eqref{gbf}) for all functions in the Sobolev space $H^s$, provided that $s > s_0$.
\begin{lemma}\label{ulc}
	Let $p \ge 1$ and $s_0 \geq0.$ Suppose  for any $\varepsilon > 0$, there is a constant $C_\varepsilon$ such that
	\begin{equation}\label{AR stri}
		\left\| \sup_{0 < t \le 1} \left| e^{it\Delta} f \right| \right\|_{L^p(B^n(0,1) \times \mathbb{T}^m)}
		\le C_\varepsilon N^{s_0 + \varepsilon} \|f\|_{L^2(\mathbb{R}^n \times \T^m)} ,
	\end{equation}
	holds for any $N \ge 1$, and any function $f \in L^2(\mathbb{R}^n \times \T^m)$ with $\operatorname{supp} \hat{f} \subset A_{\R^n \times \T^m}(N).$ Then
	\begin{equation}\label{gbf}
		\left\| \sup_{0 < t \le 1} \left| e^{it\Delta} f \right| \right\|_{L^p(B^n(0,1)\times \mathbb{T}^m)}
		\le C_s \|f\|_{H^s(\mathbb{R}^n \times \mathbb{T}^m)}
	\end{equation}
	holds for any function $f \in H^s(\mathbb{R}^n \times \T^m)$ with $s > s_0$.
\end{lemma}

\begin{proof}
	Let $s > s_0,$  $f \in H^s(\mathbb{R}^n \times \mathbb{T}^m)$, we choose $\varepsilon > 0$ such that $s > s_0 + \varepsilon$. We decompose $f$ in a Littlewood--Paley decomposition:
	\[
	f = \sum_{N \in 2^{\mathbb{N}}} P_Nf.
	\]
	By the hypothesis, we  have 
	\begin{equation}\label{f0 stri}
		\left\| \sup_{0 < t \le 1} \left| e^{it\Delta} P_1f \right| \right\|_{L^p(B^n(0,1) \times \mathbb{T}^m)}
		\lesssim \|P_1f\|_{L^2(\R^n \times \T^m)}.
	\end{equation}
	Simple calculation gives
	\[
	\|P_Nf\|_{L^2(\R^n \times \T^m)} \lesssim N^{-s} \|f\|_{H^s(\R^n \times \T^m)}.
	\]
	Applying \eqref{AR stri} to each $P_Nf$,  and using the triangle inequality, we get
	\begin{align*}
		\left\| \sup_{0 < t \le 1} \left| e^{it\Delta} f \right| \right\|_{L^p(B^n(0,1) \times \mathbb{T}^m)}
		&\lesssim \sum_{N \in 2^{\mathbb{N}},\, N>1} N^{s_0 + \varepsilon - s} \|f\|_{H^s(\R^n \times \T^m)}
		+ \|P_1f\|_{L^2(\R^n \times \T^m)} \\
		&\lesssim \|f\|_{H^s(\R^n \times \T^m)},
	\end{align*}
	as desired.
\end{proof}

Next \Cref{application of maximal strichartz} states that  maximal estimate for the Schr\"odinger evolution on waveguide manifold implies almost everywhere pointwise convergence to the initial data. Our method of proof is inspired from the Euclidean case, see \cite[Theorem 5]{Sjolin1987regularity}. Indeed, this follows from a density and measure estimate approach. We first prove uniform convergence for Schwartz functions with the aid of Fourier decay, and then use approximation to prove the convergence for arbitrary $H^s$ functions. The ingredients are Chebyshev's inequality, the given maximal estimate.
\begin{lemma}\label{application of maximal strichartz}
	Suppose that for some $s > 0$, and some $p \ge 1$,
	\begin{equation}\label{max Hs wave}
		\left\| \sup_{0 < t \le 1} |e^{it\Delta} f| \right\|_{L^p(B^n(0,1) \times \mathbb{T}^m)}
		\le C_s \|f\|_{H^s(\mathbb{R}^n \times \mathbb{T}^m)}
	\end{equation}
	holds for any function $f \in H^s(\mathbb{R}^n \times \mathbb{T}^m)$. Then
	\begin{equation}\label{pointwise as consequence of max Hs wave}
		\lim_{t \to 0} e^{it\Delta} f(x,y) = f(x,y)
		\quad \text{for a.e } (x,y) \in  \mathbb{R}^n \times \mathbb{T}^m.
	\end{equation}
\end{lemma}

\begin{proof}
	Let $\mathcal{D}(\R^n \times \T^m)$ denote the set of all smooth functions $f$ on $\R^n \times \T^m,$ satisfying
	\[
	\sup_{x\in\R^n,\, y\in\T^m}
	\langle x\rangle^{L}
	\abs{\partial_x^{a} \partial_y^{b} f(x,y)} < \infty,
	\]
	for all $L \in \mathbb{N}_0=\mathbb{N} \cup \{0\}$ and all multi-indices $a \in \mathbb{N}_0^n$, $b \in \mathbb{N}_0^m.$
	First, we show that if $f \in \mathcal{D}(\R^n \times \T^m),$ then $e^{it\Delta} f(x,y) \to f(x,y)$ uniformly on $\mathbb{R}^n \times \mathbb{T}^m$, as $t \to 0.$ Note that
	\begin{align*}
		|e^{it\Delta} f(x,y) - f(x,y)|
		&= \left| \sum_{k \in \Z^m} \int_{\R^n} 
		e^{2\pi i (x\cdot \xi + y\cdot k)} 
		\widehat{f}(\xi,k)\big(e^{2\pi i t(|\xi|^2 + |k|^2)} - 1\big)\, d\xi \right| \\
		&\lesssim |t| \sum_{k \in \Z^m} \int_{\R^n} |\widehat{f}(\xi,k)|\, (|\xi|^2 + |k|^2) \, d\xi \\
		&\le |t| \,
		\norm{\widehat{f}(\xi,k)(|\xi|^2 + |k|^2)(1+|\xi|^2+|k|^2)^d}_{L^{\infty}} \\
		&\quad \times
		\sum_{k \in \Z^m} \int_{\R^n} (1+|\xi|^2+|k|^2)^{-d} \, d\xi .
	\end{align*}
	For $f \in \mathcal{D}(\R^n \times \T^m),$ its Fourier transform $\widehat{f}$ is also a Schwartz function in the first component and has faster decay than any polynomial. Hence,
	\[
	\|\widehat{f}(\xi,k)(|\xi|^2+|k|^2) (1+|\xi|^2+|k|^2)^d\|_{L^\infty (\R^n \times \Z^m)} < \infty.
	\]
	
	Thus, it follows that $|e^{it\Delta} f(x,y) - f(x,y)| \le C |t|$. Therefore, the uniform convergence is justified for Schwartz functions on $\mathbb{R}^n \times \mathbb{T}^m$.
	
	Given any function $f \in H^s(\mathbb{R}^n \times \mathbb{T}^m)$, since $\mathcal{D}(\R^n \times \T^m)$ is dense in $H^s(\R^n \times \T^m)$, for any $\varepsilon > 0$, we write
	\[
	f = g + h,
	\]
	where $g$ is Schwartz and $\|h\|_{H^s(\mathbb{R}^n \times \mathbb{T}^m)} < \varepsilon$. Then we get
	\begin{align*}
		\limsup_{t\to 0} \abs{e^{it\Delta} f(x,y) - f(x,y)}
		&\le \limsup_{t\to 0} \abs{e^{it\Delta} g(x,y) - g(x,y)} \\
		&\quad + \limsup_{t\to 0}\abs{e^{it\Delta} h(x,y) - h(x,y)} \\
		&\le 0 + \sup_{0<t\le1} \abs{e^{it\Delta} h(x,y)} + \abs{h(x,y)}.
	\end{align*}
	Let $X = B^n(0,1)\times \T^m.$ For any $\alpha > 0,$ denote
	\[
	E_\alpha := \left\{ (x,y) \in X : \limsup_{t\to 0} \abs{e^{it\Delta} f(x,y) - f(x,y)} > \alpha \right\}.
	\]
	We bound $|E_\alpha|$, the Lebesgue measure of $E_\alpha$, by
	\[
	\left| \left\{ (x,y) \in X : \sup_{0<t\le1} \abs{e^{it\Delta} h(x,y)} > \alpha/2 \right\} \right|
	+ \left| \left\{ (x,y) \in X : \abs{h(x,y)} > \alpha/2 \right\} \right|.
	\]
	Moreover by Chebyshev's inequality and assumption \eqref{max Hs wave}, we have
	\begin{align*}
		\left| \left\{ (x,y) \in X : \sup_{0<t\le1} \abs{e^{it\Delta} h(x,y)} > \alpha/2 \right\} \right|
		&\le \frac{\norm{\sup_{0<t\le1} \abs{e^{it\Delta} h}}_{L^p(X)}^p}{(\alpha/2)^p} \\
		&\le \frac{C_s^p \norm{h}_{H^s(\R^n \times \T^m)}^p}{(\alpha/2)^p}
		\lesssim \frac{\varepsilon^p}{\alpha^p},
	\end{align*}
	and
	\begin{align*}
		\left| \left\{ (x,y) \in X : \abs{h(x,y)} > \alpha/2 \right\} \right|
		&\le \frac{\norm{h}_2^2}{(\alpha/2)^2}
		\le \frac{\norm{h}_{H^s(\R^n \times \T^m)}^2}{(\alpha/2)^2}
		\lesssim \frac{\varepsilon^2}{\alpha^2}.
	\end{align*}
	Therefore, for any $\alpha > 0$ the bound
	\[
	|E_\alpha| \lesssim \frac{\varepsilon^p}{\alpha^p} + \frac{\varepsilon^2}{\alpha^2},
	\]
	holds for any $\varepsilon > 0$. We conclude that $|E_\alpha| = 0$ for any $\alpha > 0$. We take the set $\bigcup_{k=1}^{\infty} E_{1/k}$, which has measure zero and
	\[
	\lim_{t\to0} e^{it\Delta} f(x,y) = f(x,y)
	\quad \text{for any } (x,y) \in X \setminus \bigcup_{k=1}^{\infty} E_{1/k}.
	\]
	We have proved the almost everywhere pointwise convergence in $X$. The same argument applies to any other set of this kind $B^n(x_0,1)\times \T^m$, via the following observation: we write $x = x_0 + x_1 \in B^n(x_0,1)$, where $x_1 \in B^n(0,1).$ Indeed, keeping $y \in \T^m$ fixed,  we have 
	\[
	e^{it\Delta} f(x,y)
	= \sum_{k\in\Z^m} \int_{\R^n} 
	e^{2\pi i (x_1\cdot \xi + y\cdot k + t(|\xi|^2+|k|^2))} 
	e^{2\pi i x_0\cdot \xi} \widehat{f}(\xi,k)\, d\xi
	=: e^{it\Delta} g(x_1,y),
	\]
	where $\widehat{g}(\xi,k) = e^{2 \pi ix_0\cdot \xi} \widehat{f}(\xi,k)$.  Hence,
	\[
	\norm{ \sup_{0<t\le1} \abs{e^{it\Delta} f} }_{L^p(B^n(x_0,1)\times \T^m)}
	=
	\norm{ \sup_{0<t\le1} \abs{e^{it\Delta} g} }_{L^p(X)}
	\le C_s \norm{g}_{H^s(\R^n \times \T^m)}
	= C_s \norm{f}_{H^s(\R^n \times \T^m)}.
	\]
	Therefore, $\lim_{t \to 0} e^{it\Delta} f(x,y) = f(x,y)$ for almost every $(x,y) \in \R^n \times \T^m$.
\end{proof}

\subsection{Lieb--Sobolev inequality on waveguide manifold} 
In this subsection, we aim to prove the following Lieb--Sobolev inequality (\Cref{lem:lieb-sobolev-product}) on the waveguide manifold. Indeed, this  is an orthonormal (density-matrix) analogue of the classical Sobolev embedding 
\begin{equation}\label{soem}
	H^s(\M)\hookrightarrow L^{2p}(\M) \quad \text{for} \ s \geq \frac{d}{2p'}, \ \  s  \in \left(0,\frac{d}{2} \right)
\end{equation}
It bounds, in a single $L^p$ norm, the density built from an entire orthonormal family $(f_j)_j$ weighted by $\lambda\in\ell^{p,1}$ (Lorentz sequence space, see \eqref{Lorentz sequence space}),  rather than a single function. We use it in \Cref{sec7} to control  the  sequence of density functions (see \eqref{e:lieb-sobolev}). Hence to define the density functions $\rho_{\gamma_0}$ and $\rho_{\gamma(t)}$ for general $\gamma_0\in\Sp^{\alpha',s}$, and  not merely finite-rank $\gamma_0$. 
Inequalities of this type for orthonormal systems trace back to Lieb \cite{Lieb1983orthonormalsobolev}.  See also Frank--Sabin \cite{frank2017restriction}.

\begin{proposition}[Lieb--Sobolev inequality on $\M$]\label{lem:lieb-sobolev-product}
	Let $d\ge 1,$ $1<p<\infty$ and $\frac{d}{2p'} \leq s < \frac{d}{2}$. Then
	\[
	\left\| \sum_j \lambda_j \left| \langle \nabla \rangle^{-s} f_j \right|^2 \right\|_{L^p(\M)}
	\lesssim \norm{\lambda}_{\ell^{p,1}}
	\]
	holds for every orthonormal system $(f_j)_j$ in $L^2(\M)$ and every $\lambda=(\lambda_j)_j \in \ell^{p,1}$.
\end{proposition}
\begin{remark}
	Taking $j=1$, $\lambda_1=1$, and $\lambda_j=0$ for $j>1$ in \Cref{lem:lieb-sobolev-product}, with $g=\langle\nabla\rangle^{-s}f$, gives $\|g\|_{L^{2p}}\lesssim1$. Since $\|f\|_{L^2}=\|g\|_{H^s}=1$, this recovers \eqref{soem}.
\end{remark}
In order to prove \Cref{lem:lieb-sobolev-product}, we briefly set some notations and recall 
operator-valued Cwikel theorem. Let sigma-finite measure spaces $(Y_1,dy_1)$, $(Y_2,dy_2)$ and separable Hilbert spaces $\cH$, $\mathcal{G}$.   For $p \ge 1$, write $L^p(Y_1,\Sp^p(\cH))$ for the measurable $\cH$-compact-operator-valued functions $f$ on $Y_1$ satisfying
\[
\|f\|^p_{L^p(\Sp^p)} = \int_{Y_1} \|f(y_1)\|^p_{\Sp^p(\cH)}\, dy_1 < \infty.
\]
See \Cref{schp} for Schatten space $\Sp^p.$
For the weak-type space, $L^{p,\infty}(Y_2,\mathcal{B}(\mathcal{G}))$  consists of measurable functions $g$ on $Y_2$ valued in bounded operators on $\mathcal{G}$ for which
\[
\|g\|^p_{L^{p,\infty}(\mathcal{B})} = \sup_{\tau>0} \tau^p \left|\{y_2 \in Y_2 : \|g(y_2)\|_{\mathcal{B}(\mathcal{G})} > \tau\}\right| < \infty.
\]
The following theorem bounds the weak Schatten norm of the product \( f\Phi^*g \), which is constructed using multiplication operators \( f \) and \( g \), along with a unitary operator \( \Phi \) that relates \( L^2(Y_1,\cH) \) to \( L^2(Y_2,\mathcal G) \). This theorem expresses the weak Schatten norm in terms of the \( L^q(\Sp^q) \) and weak \( L^q \) norms of \( f \) and \( g \). We will apply it below with \( \Phi = \mathcal F \) (Fourier transform) to derive the Cwikel estimate mentioned in \Cref{lem:cwikel-estimate}.
\begin{knowntheorem}[Operator-valued Cwikel theorem \cite{Frank2014Cwikel}]\label{Cwikel operator valued thm}
	Suppose
	\[
	\Phi : L^2(Y_1,\cH) \longrightarrow L^2(Y_2,\mathcal{G})
	\]
	is unitary and  also a bounded map $L^1(Y_1,\cH) \to L^\infty(Y_2,\mathcal{G}).$  Fix $q>2$. For $f \in L^q(Y_1,\Sp^q(\cH))$ and $g \in L^{q,\infty}(Y_2,\mathcal{B}(\mathcal{G}))$, 
	\[
	f \Phi^* g \in \Sp_{q,\infty}\big(L^2(Y_2,\mathcal{G}), L^2(Y_1,\cH)\big),
	\]
	and
	\begin{equation}\label{Cwikel operator valued eq}
		\|f \Phi^* g\|^q_{\Sp_{q,\infty}\big(L^2(Y_2,\mathcal{G}), L^2(Y_1,\cH)\big)} \le \frac{q}{2}\left(\frac{q}{q-2}\right)^{q-1} C^2 \, \|f\|^q_{L^q(\Sp^q)} \, \|g\|^q_{L^{q,\infty}(\mathcal{B})},
	\end{equation}
	where $f$ acts as a multiplication operator on the left-hand side, $\Phi^*$ denotes the adjoint of $\Phi$, and $C = \|\Phi\|_{L^1(Y_1,\cH) \to L^\infty(Y_2,\mathcal{G})}$.
\end{knowntheorem}

Take $Y_1=\R^n\times\T^m$ and $Y_2=\R^n\times\Z^m$, and let $\cH=\mathcal{G}=\mathbb{C}$, so that \Cref{Cwikel operator valued thm} becomes a statement about scalar functions. Let $\Phi=\mathcal{F}$, the Fourier transform on $\R^n\times\T^m$. Plancherel's theorem tells us that $\mathcal{F}$ is unitary from $L^2(\R^n\times\T^m)$ onto $L^2(\R^n\times\Z^m)$, and the usual $L^1$-$L^\infty$ bound for the Fourier transform gives
\[
\|\mathcal{F}f\|_{L^\infty(\R^n\times\Z^m)}
\leq
\|f\|_{L^1(\R^n\times\T^m)}.
\]
So $\Phi$ satisfies both hypotheses of the theorem, with $C = \|\mathcal{F}\|_{L^1(\R^n\times\T^m)\to L^\infty(\R^n\times\Z^m)} \le 1$.
Since $\Phi=\mathcal{F}$ is a unitary operator and weak Schatten norms are unchanged by composing with $\Phi$ on the right, in particular $$\|V\Phi^*g\|_{\Sp_{q,\infty}(L^2(\R^n\times\Z^m),\,L^2(\R^n\times\T^m))}=\|(V\Phi^*g)\Phi\|_{\Sp_{q,\infty}(L^2(\R^n\times\T^m),\,L^2(\R^n\times\T^m))}$$ and $$(V\Phi^*g)\Phi=V\mathcal{F}^{-1}g\mathcal{F}=V\langle\nabla\rangle^{-s}.$$ This allows us to move back and forth easily between the two presentations below.

Now take
\[
g(\xi,k)
=
\bigl(1+|\xi|^2+|k|^2\bigr)^{-s/2},
\]
which is exactly the multiplier for the operator $\langle \nabla \rangle^{-s}$, i.e.
\[
\langle \nabla \rangle^{-s}=\mathcal{F}^{-1}g\,\mathcal{F}.
\]
Plugging $f=V$ and this $g$ into \eqref{Cwikel operator valued eq} gives 
\[
\|V\langle \nabla \rangle^{-s}\|_{\Sp_{q,\infty}(L^2(\R^n\times\T^m),\,L^2(\R^n\times\T^m))}
\lesssim
\|V\|_{L^q(\R^n\times\T^m)}
\|g\|_{L^{q,\infty}(\R^n\times\Z^m)}.
\]
And if $q=\frac{n+m}{s}>2,$ then $g$ lies in $L^{q,\infty}(\R^n\times\Z^m)$ for the right range of $q$.
i.e.\ whenever $sq=n+m$ with $q>2$. Under this condition $\|g\|_{L^{q,\infty}(\R^n\times\Z^m)}$ is finite, so
\[
\|V\langle \nabla \rangle^{-s}\|_{\Sp_{q,\infty}(L^2(\R^n\times\T^m),\,L^2(\R^n\times\T^m))}
\lesssim
\|V\|_{L^q(\R^n\times\T^m)},
\qquad
q \geq \frac{n+m}{s}>2.
\]

\begin{lemma}[Cwikel estimate]\label{lem:cwikel-estimate}
	Let $d\ge 1,$ $q>2$ and $\frac{d}{q} \leq s < \frac{d}{2}$. Then for all $V \in L^q(\M)$,
	\[
	\|V\langle \nabla \rangle^{-s}\|_{\Sp_{q,\infty}(L^2(\M),\,L^2(\M))} \lesssim \|V\|_{L^q(\M)}.
	\]
\end{lemma}
\begin{proof}
For $\M=\R^n\times\T^m$, this is exactly the computation carried out above: apply
\Cref{Cwikel operator valued thm} with $Y_1=\R^n\times\T^m$, $Y_2=\R^n\times\Z^m$,
$\cH=\mathcal G=\mathbb C$, $\Phi=\mathcal F$ and $f=V$, and take
$g(\xi,k)=(1+|\xi|^2+|k|^2)^{-s/2}$, which lies in $L^{q,\infty}(\R^n\times\Z^m)$
precisely when $s\ge d/q$. The case $\M=\T^d$ is identical with $Y_2=\Z^d$.
For compact manifolds the estimate is also available from
Sukochev--Yang--Zanin \cite[Lemma 2.3]{sukochev2026singular}.
\end{proof}

We are now ready to prove \Cref{lem:lieb-sobolev-product}.
\begin{proofof}{lem:lieb-sobolev-product}
	Set
	\[
	G = \sum_j \lambda_j \left| \langle \nabla \rangle^{-s} f_j \right|^2.
	\]
	We want to bound $\norm{G}_{L^p(\M)}$, and for this it is enough, by duality, to show
	\begin{equation}\label{eq:duality-goal}
		\left| \int_{\M} G(x)V(x)\, dx \right| \lesssim \norm{\lambda}_{\ell^{p,1}}
	\end{equation}
	for every positive $V\in L^{p'}(\M)$ with $\norm{V}_{L^{p'}(\M)} = 1$.
	 Let
	\[
	K = V^{1/2} \langle \nabla \rangle^{-s}, \qquad K^* = \langle \nabla \rangle^{-s} V^{1/2}, \qquad L = K^*K.
	\]
	Let $\gamma$ be the diagonal operator built from $\lambda=(\lambda_j)_j$ and $(f_j)_j$,
	\[
	\gamma = \sum_j \lambda_j \lvert f_j \rangle \langle f_j \rvert, \qquad \lvert f \rangle \langle g \rvert (h) := \langle g,h\rangle f,
	\]
	so that by construction
	\begin{equation}\label{eq:gamma-norm}
		\norm{\gamma}_{\Sp_{p,1}} = \norm{\lambda}_{\ell^{p,1}}.
	\end{equation}
	Expanding the left side of \eqref{eq:duality-goal} term by term,
	\[
	\int_{\M} G(x)V(x)\,dx = \sum_j \lambda_j \int_{\M} \left| V^{1/2}(x)\langle \nabla \rangle^{-s} f_j(x) \right|^2 dx
	= \sum_j \lambda_j \langle Kf_j, Kf_j\rangle_{L^2(\M)},
	\]
	and since $\langle Kf_j,Kf_j\rangle = \langle f_j, K^*Kf_j\rangle = \langle f_j, Lf_j\rangle$,
	\[
	\int_{\M} G(x)V(x)\,dx = \sum_j \lambda_j \langle f_j, Lf_j \rangle_{L^2(\M)} = \Tr(\gamma L).
	\]
	Now we bound $K$. Since $V\in L^{p'}(\M)$, we have $V^{1/2}\in L^{2p'}(\M)$ with
	\[
	\norm{V^{1/2}}_{L^{2p'}(\M)} = \norm{V}_{L^{p'}(\M)}^{1/2} = 1.
	\]
	Since $s \ge \dfrac{d}{2p'} := \dfrac{d}{q}$, by \Cref{lem:cwikel-estimate}, we get 
	\begin{equation}\label{eq:K-bound}
		\norm{K}_{\Sp_{q,\infty}} = \left\| V^{1/2}\langle \nabla \rangle^{-s} \right\|_{\Sp_{q,\infty}}
		\lesssim \norm{V^{1/2}}_{L^q(\M)} = \norm{V}_{L^{p'}(\M)}^{1/2} = 1.
	\end{equation}
	The singular values of $K$ and $L=K^*K$ are related by $s_n(L) = s_n(K)^2$, and since $q=2p'$ this turns the weak $L^q$ bound on $K$ into a weak $L^{p'}$ bound on $L$:
	\begin{equation}\label{eq:L-bound}
		\norm{L}_{\Sp_{p',\infty}} = \norm{K}_{\Sp_{q,\infty}}^2 \lesssim 1.
	\end{equation}
	Combining \eqref{eq:gamma-norm} and \eqref{eq:L-bound} with the Hölder inequality for Lorentz--Schatten norms ($\frac1p+\frac1{p'}=1$),
	\[
	\left| \Tr(\gamma L) \right| \le \norm{\gamma}_{\Sp_{p,1}} \norm{L}_{\Sp_{p',\infty}}
	\lesssim \norm{\lambda}_{\ell^{p,1}}.
	\]
	This is exactly \eqref{eq:duality-goal}, so the lemma follows.
\end{proofof}

\section{Sufficient condition for single equation}\label{sec5}    

In this section, we shall prove \Cref{max stri for wave} and  \Cref{genind}.
\begin{remark}[proof strategy]\label{pst1.1}
	 A Sobolev-type pointwise bound in $t$ (\Cref{calculus lemma}) reduces $\sup_t|e^{it\Delta}P_Nf|$ to an $L^p_{x,y,t}$ norm, controlled by Bernstein's inequality \eqref{Bernstein} and the diagonal Strichartz estimate of \Cref{diag strichartz}. Then \Cref{application of maximal strichartz}  converts the resulting maximal estimate into a.e. convergence via density and Chebyshev inequality argument.
\end{remark}
\begin{proofof}{max stri for wave}
	
	Let $d=n+m.$ We actually prove the stronger estimate
	\[
	\left\|
	\sup_{0\le t\le 1}
	\left|e^{it\Delta}P_Nf(x,y)\right|
	\right\|_{L^{\frac{2(d+2)}{d}}_{x,y}(\mathbb{R}^n\times\mathbb{T}^m)}
	\lesssim_\varepsilon
	N^{\frac{d}{d+2}+\varepsilon}
	\|f\|_{L^2(\mathbb{R}^n\times\mathbb{T}^m)}.
	\]
	Put $F(t)=e^{it\Delta}P_Nf,$ then by \eqref{NLS},
	\[
	F'(t)
	=
	i\Delta e^{it\Delta}P_Nf.
	\]
	Applying  \Cref{calculus lemma} to the above $F$ with $a=0,b=1$,  we obtain
	\[
	\sup_{0\le t\le 1}
	\left|e^{it\Delta}P_Nf(x,y)\right|
	\lesssim
	|P_Nf(x,y)|
	+
	\mu^{\frac1p-1}
	\|\Delta e^{it\Delta}P_Nf\|_{L_t^p([0,1])}
	+
	\mu^{\frac1p}
	\|e^{it\Delta}P_Nf\|_{L_t^p([0,1])}.
	\]
	Since operators $P_N$ and $e^{it\Delta}$ commutes, so using \eqref{Bernstein} (Bernstein’s inequality),
	\[
	\|\Delta e^{it\Delta}P_Nf\|_{L_t^p([0,1])}
	\lesssim
	N^2
	\|e^{it\Delta}P_Nf\|_{L_t^p([0,1])}.
	\]
	Hence,
	\[
	\sup_{0\le t\le 1}
	\left|e^{it\Delta}P_Nf(x,y)\right|
	\lesssim
	|P_Nf(x,y)|
	+
	\mu^{\frac1p-1}N^2
	\|e^{it\Delta}P_Nf\|_{L_t^p([0,1])}
	+
	\mu^{\frac1p}
	\|e^{it\Delta}P_Nf\|_{L_t^p([0,1])}.
	\]
	Taking $\mu=N^{2}$ and $p=\frac{2(d+2)}{d}$, we obtain
	\[
	\sup_{0\le t\le 1}
	\left|e^{it\Delta}P_Nf(x,y)\right|
	\lesssim
	|P_Nf(x,y)|
	+
	2N^{\frac{2}{p}}
	\|e^{it\Delta}P_Nf\|_{L_t^p([0,1])}.
	\]
	Taking the \(L^p(\mathbb{R}^n\times\mathbb{T}^m)\)-norm on both sides and Minkowski's inequality,
	\[
	\left\|
	\sup_{0\le t\le 1}
	\left|e^{it\Delta}P_Nf(x,y)\right|
	\right\|_{L^p(\mathbb{R}^n\times\mathbb{T}^m)}
	\lesssim
	\|P_Nf\|_{L^p(\mathbb{R}^n\times\mathbb{T}^m)}
	+
	2N^{\frac{2}{p}}
	\|e^{it\Delta}P_Nf\|_{L^p_{x,y,t}
		(\mathbb{R}^n\times\mathbb{T}^m\times[0,1])}.
	\]
	Using \eqref{Bernstein}  and \Cref{diag strichartz}, we have 
	\begin{align*}
		\left\|
		\sup_{0\le t\le 1}
		\left|e^{it\Delta}P_Nf\right|
		\right\|_{L^p(\mathbb{R}^n\times\mathbb{T}^m)} &\lesssim_\varepsilon
		N^{\frac d2-\frac d p}
		\|P_Nf\|_{L^2(\mathbb{R}^n\times\mathbb{T}^m)}+
		2N^{\frac d{d+2}}
		N^{\varepsilon}
		\|f\|_{L^2(\mathbb{R}^n\times\mathbb{T}^m)} \\
		&\lesssim_\varepsilon
		N^{\frac d{d+2} +\varepsilon}
		\|f\|_{L^2(\mathbb{R}^n\times\mathbb{T}^m)}.
	\end{align*}
	Since $p \geq2$ and $L^p_{x,y} \left(\R^n \times \mathbb{T}^m \right) \hookrightarrow L^p_{x,y} \left(B^n(0,1)\times \mathbb{T}^m \right) \hookrightarrow L^2_{x,y} \left(B^n(0,1)\times \mathbb{T}^m \right),$
	\begin{equation*}
		\| \sup_{0 < t \leq 1} \left| e^{it\Delta} P_N f(x,y) \right| \|_{L^2_{x,y} \left(B^n(0,1)\times \mathbb{T}^m \right)}
		\lesssim_\varepsilon N^{\frac{n+m}{n+m+2}+\varepsilon}
		\|f\|_{L^2(\mathbb{R}^n \times \mathbb{T}^m)}
	\end{equation*}
\end{proofof}

\begin{proof}[\textbf{Proof of \Cref{genind}}]  \Cref{max stri for wave}  ensures the hypothesis of \Cref{ulc} is satisfied to  get 
	\eqref{gbf} with $s_0=\frac{n+m}{n+m+2}$, $p=2$. But this means  $f$ satisfies the maximal estimate \eqref{max Hs wave} with this $s$ and $p=2$. Then  \Cref{application of maximal strichartz}  gives  the desired point wise  convergence.
\end{proof}

\section{Necessary condition for single equation}\label{sec-necc-single}
In this section we shall prove \Cref{Necessary condition for theorem}.
Our method of proof is inspired by the Bourgain's counterexample \cite{Bourgain2016note} for Schr\"odinger maximal function,  which was done in the Euclidean space $\mathbb R^n$. For detailed exposure on this, we refer to  Pierce \cite{Pierce2020BourgainCounter}. We adopt this  approach in the setting of the waveguide manifold  $\mathbb R^n \times \mathbb T^m.$

\begin{proposition}\label{Necessary condition for proposition}
	Let $n,m\geq 1$ and  $s<\frac{n+m}{2(n+m+1)}.$  Then
	there exist a sequence $R_k\to\infty$ and functions
	$F_k\in L^2(\R^n\times\T^m)$ satisfying
	$\norm{F_k}_{L^2(\R^n\times\T^m)}=1,$
	whose Fourier transforms are supported in the frequency annulus
	\[
	\left\{(\xi,\ell)\in\R^n\times\Z^m:
	\abs{\xi}^2+\abs{\ell}^2\sim R_k^2\right\},
	\]
	such that
	\[
	\lim_{k\to\infty}
	R_k^{-s}
	\norm{
		\sup_{0<t<1}
		\abs{e^{it\Delta}F_k}
	}_{L^1(B(0,1)\times\T^m)}
	=\infty.
	\]
\end{proposition}

\begin{remark}  Our construction of $F_k$ in \Cref{Necessary condition for proposition} adapts the approach used by Pierce \cite{Pierce2020BourgainCounter} for $\mathbb R^n$. However, due to  mixed geometric features of waveguide manifold $\R^n \times \T^m$, we need to choose torus component carefully. Indeed, we need to localize frequencies simultaneously in a Euclidean direction and in periodic directions. More specifically, we briefly highlight the differences.
	\begin{itemize}
		\item[--]  We keep the first coordinate $x_1$ continuous, exactly as in the Euclidean construction, with frequency concentrated near $R$ at scale $R^{1/2}$ (see $f(x_1)$ in \eqref{eq:def-f}). In the remaining $(n-1)$-Euclidean coordinates $x'$ and in the periodic coordinate $y\in\T^m$, we replace the continuous frequency localization by a lattice sum over $\ell=(\ell_1,\ell_2)\in\Lambda_1\times\Lambda_2$, using a single integer dilation parameter $D$. This forces $D\ell_2\in\Z^m$ so that each summand is a character on $\T^m$ (see $g(x')h(y)$ in \eqref{eq:def-f}), which has no analogue on $\R^n$. 
		\item[--] The waveguide structure  $X_R$ (to be defined below \eqref{dfxr}) is different as compared to the Euclidean case.  We shall  need a lower bound $|X_R|\gtrsim1$ for the resonant set. Consequently, the rescaling map $T$ acts differently on the Euclidean and periodic coordinates: on $x_1,x'$ it is an ordinary linear change of variables, while on $y$ the integer dilation $y\mapsto Dy$ is compatible with the $1$-periodicity of $\T^m$ and produces $D^m$ periodic copies of the fundamental domain. Accounting for this mixed Jacobian, and then applying \Cref{lem:A.1} to control the resulting union of rational neighborhoods on the periodic directions, is what let us close the argument on $\R^n\times\T^m$.
	\end{itemize}  
\end{remark}
\begin{remark}[notations]\label{nota} For the convenience in the proof of  \Cref{Necessary condition for proposition}, we   fix some notations. 
	\begin{itemize}
		\item Let $(x,y)\in\R^n\times\T^m,$ where $x=(x_1,\dots,x_n)=(x_1,x')\in\R^n, $ $ x'=(x_2,\dots,x_n)\in\R^{n-1},$ and $y=(y_1,\dots,y_m)\in\T^m.$
		\item We write $(\xi,k)\in\R^n\times\Z^m,$ where $\xi=(\xi_1,\dots,\xi_n)=(\xi_1,\xi')\in\R^n,$ $ \xi'=(\xi_2,\dots,\xi_n)\in\R^{n-1},$ and $k=(k_1,\dots,k_m)\in\Z^m.$
	\end{itemize}
\end{remark}
Next, before proving \Cref{Necessary condition for proposition}, we  construct a test function $F$. This will be  used throughout the proof.
\begin{itemize}
	\item[--] \underline{\textbf{\textit{Construction of the test function $F$}.\\}}
\end{itemize}
Fix Schwartz functions $\phi:\R\to\R_+$ and $\Phi:\R^{n-1}\to\R_+$ satisfying
\[
\operatorname{supp}\widehat\phi\subset[-1,1],\qquad \operatorname{supp}\widehat\Phi\subset B^{n-1}(0,1)\quad \text{and} \quad \phi(0)=\Phi(0)=1,
\]
with $$\Phi(x')=\prod_{i=2}^n\varphi(x_i),$$ so that $$\widehat\Phi(\xi')=\prod_{i=2}^n\widehat\varphi(\xi_i).$$  For the  precise  construction of such Schwartz functions $\phi$ and $\varphi$, we refer to  Pierce \cite[Sect. 2.1]{Pierce2020BourgainCounter}. Let $R$ and $D$ be two parameters, where $D$ is a positive integer that depends on $R$. The appropriate values for these parameters will be chosen later.

We define the frequency-localized lattice sets as follows:
\begin{align}
	\Lambda_1&=\Big\{\ell_1=(\ell_{1,1},\dots,\ell_{1,n-1})\in\Z^{n-1}\ :\ \frac{R}{2D}<\ell_{1,i}<\frac RD,\ i=1,\dots,n-1\Big\},\notag\\
	\Lambda_2&=\Big\{\ell_2=(\ell_{2,1},\dots,\ell_{2,m})\in\Z^{m}\ :\ \frac{R}{2D}<\ell_{2,j}<\frac RD,\ j=1,\dots,m\Big\}.\notag
\end{align}
Denote $$\Lambda:=\Lambda_1\times\Lambda_2= \{\ell=(\ell_1,\ell_2): \ell_1\in \Lambda_1, \ell_2 \in \Lambda_2\}.$$ By construction,
\begin{equation}\label{eq:lattice-size}
	\abs{\Lambda_1}\sim(R/D)^{n-1},\qquad \abs{\Lambda_2}\sim(R/D)^{m},\qquad \abs{\Lambda}\sim(R/D)^{n+m-1}.
\end{equation}
Here $\abs{\Lambda}$ denotes cardinality of the set $\Lambda.$
Define
\begin{equation}\label{eq:def-F}
	F(x,y)=e^{2\pi iRx_1}\,\phi(R^{1/2}x_1)\,\Phi(x')\sum_{\ell\in\Lambda}e^{2\pi iD\ell_1\cdot x'}\,e^{2\pi iD\ell_2\cdot y}.
\end{equation}
Here $R$   will be chosen in a way so that   $D\ell_2\in\Z^m$ for every $\ell_2\in\Lambda_2$. See \eqref{choosing R_k}.   Notice that  each summand is a well-defined character on $\T^m$, and $F$ is a well-defined Schwartz function on $\R^n\times\T^m$.\\
\noindent
\begin{itemize}
	\item[--] \underline{\textit{\textbf{Computation of the $L^2$ norm of $F$}}}. 
\end{itemize} 
We may factor $F$  as product of functions in $x_1, x'$ and $y,$ specifically, rewrite $F$ as follows  
\begin{equation*}
	F(x,y)=f(x_1)g(x')h(y)
\end{equation*}
with
\begin{equation}\label{eq:def-f}
    f(x_1)=e^{2\pi iRx_1}\phi(R^{1/2}x_1),\qquad
g(x')=\Phi(x')\sum_{\ell_1\in\Lambda_1}e^{2\pi iD\ell_1\cdot x'},\qquad
h(y)=\sum_{\ell_2\in\Lambda_2}e^{2\pi iD\ell_2\cdot y}.
\end{equation}
Substituting $u=R^{1/2}x_1$, we have 
\begin{equation}\label{eq:norm-g}
	\norm{f}_{L^2(\R)}^2=\int_\R\abs{\phi(R^{1/2}x_1)}^2\,dx_1=R^{-1/2}\norm{\phi}_{L^2}^2\sim R^{-1/2}.
\end{equation}
Since $\widehat\Phi$ is supported in $B^{n-1}(0,1)$ and the functions $\Phi(x')e^{2\pi iD\ell_1\cdot x'}$, $\ell_1\in\Lambda_1$, have Fourier transforms $\widehat\Phi(\xi'-D\ell_1)$ supported in pairwise disjoint balls of radius $1$ centered at $D\ell_1.$ Distinct $\ell_1\in\Lambda_1\subset\Z^{n-1}$ give centers at least $D$ apart, and $D>2$. So these balls are pairwise disjoint. Disjoint Fourier support gives $L^2$-orthogonality by Plancherel, and it follows that 
\begin{equation}\label{eq:norm-h}
	\norm{g}_{L^2(\R^{n-1})}^2=\sum_{\ell_1\in\Lambda_1}\norm{\Phi}_{L^2}^2=\abs{\Lambda_1}\,\norm{\Phi}_{L^2}^2\sim (R/D)^{n-1}.
\end{equation}
On $\T^m$, distinct $\ell_2\in\Z^m$ give distinct dual frequencies $D\ell_2\in\Z^m$, so the characters $e^{2\pi iD\ell_2\cdot y}$ are pairwise orthonormal in $L^2(\T^m)$. Hence
\begin{equation}\label{eq:norm-k}
	\norm{h}_{L^2(\T^m)}^2=\sum_{\ell_2\in\Lambda_2}1=\abs{\Lambda_2}\sim (R/D)^{m},\qquad
	\norm{h}_{L^2(\T^m)}\sim (R/D)^{m/2}.
\end{equation}
Combining \eqref{eq:norm-g}, \eqref{eq:norm-h} and \eqref{eq:norm-k}, we compute $L^2$-norm of $F:$
\begin{equation}\label{eq:norm-f}
	\norm{F}_{L^2(\R^n\times\T^m)}=\norm{f}_{L^2}\norm{g}_{L^2}\norm{h}_{L^2}\ \sim\ R^{-1/4}\left(\frac RD\right)^{\frac{n+m-1}{2}}.
\end{equation}

\begin{itemize}
	\item[--] \underline{\textit{\textbf{Frequency localization}.}}
\end{itemize}
Recalling definitions of $f,g$ and $h$ from \eqref{eq:def-f}, we have 
\[
\widehat f(\xi_1)=\int_\R \phi(R^{1/2}x_1)\,e^{-2\pi ix_1(\xi_1-R)}\,dx_1=R^{-1/2}\,\widehat\phi\!\left(\frac{\xi_1-R}{R^{1/2}}\right),
\]
\[
\widehat g(\xi')=\sum_{\ell_1\in\Lambda_1}\widehat\Phi(\xi'-D\ell_1),\qquad
\widehat h(k)=\mathbf 1_{\{D\ell_2: \ \ell_2\in\Lambda_2\}}(k).
\]
In view of this and  \eqref{eq:def-F}, we have 
\begin{equation}\label{eq:hat-f}
	\widehat F(\xi,k)=R^{-1/2}\,\widehat\phi\!\left(\frac{\xi_1-R}{R^{1/2}}\right)\sum_{\ell_1\in\Lambda_1}\widehat\Phi(\xi'-D\ell_1)\cdot\mathbf 1_{\{D\ell_2: \ \ell_2\in\Lambda_2\}}(k).
\end{equation}

On the support of \eqref{eq:hat-f}, $\abs{\xi_1-R}\le R^{1/2}$, hence $\xi_1=R+O(R^{1/2})$. Since $\abs{\xi'-D\ell_1}\le1$ and $\abs{D\ell_1}\sim R$, the triangle inequality
\[
\big|\,\abs{\xi'-D\ell_1}-\abs{D\ell_1}\,\big|\le\abs{\xi'}\le\abs{\xi'-D\ell_1}+\abs{D\ell_1}
\]
forces $\abs{\xi'}\sim R$; likewise $\abs{k}=\abs{D\ell_2}\sim R$. Combining these estimates,
\begin{equation}\label{eq:supp-f}
	\operatorname{supp}\widehat F\ \subset\ \{(\xi,k)\ :\ \abs{(\xi,k)}\sim R\}.
\end{equation}
\begin{remark} Taking the construction of test function $F$ (see \eqref{eq:def-F}) into account, the following steps are in order to prove \Cref{Necessary condition for proposition}.
	\begin{itemize}
		\item[-]  First, we reduce $e^{it\Delta}F(x,y)$  to the quadratic Gauss sum \eqref{eq:S-estimate}. Thus, the problem is reduced to estimating the Gauss sum. At rational points, the sum can be evaluated exactly. This gives the lower bound \eqref{eq:pointwise-final} for $(x,y)$ near suitable rationals $p_1/q$ and $p'/q$.
		\item[-] Second,  we show  that  $X_R$ (\eqref{dfxr}) (arises from the previous step), has the Lebesgue measure bounded below independently of $R$. Hence, the lower bound obtained in Step I holds on a set of positive measure. Combining the two steps gives the divergence stated in \Cref{Necessary condition for proposition}.
	\end{itemize}
\end{remark}

We are now ready to prove \Cref{Necessary condition for proposition}.

\begin{proofof}{Necessary condition for proposition} We divide the proof into two steps:
	Let  $F$ be as in \eqref{eq:def-F}. 
	
	\begin{itemize}
		\item[--] \textbf{Step I: Reduction of the propagator $e^{it\Delta} F$ into Gauss sum \eqref{eq:S-estimate} and application to lower bound estimate. Specifically, for certain $(x,y)$ we have   
			\begin{equation*}
				\frac{\abs{e^{it\Delta}F(x,y)}}{\norm{F}_{L^2}}\ \gtrsim\ R^{1/4}\left(\frac{R}{Dq}\right)^{\frac{n+m-1}{2}}.
		\end{equation*}} 
		
	\end{itemize}
	By Fourier inversion and \eqref{eq:hat-f},
	\begin{equation}\label{eq:prop-1}
		e^{it\Delta}F(x,y)=\int_{\R^n}R^{-1/2}\,\widehat\phi\!\left(\frac{\xi_1-R}{R^{1/2}}\right)\sum_{\ell_1\in\Lambda_1}\widehat\Phi(\xi'-D\ell_1)\sum_{\ell_2\in\Lambda_2}e^{2\pi i\left[x\cdot\xi+y\cdot D\ell_2+\abs\xi^2t+\abs{D\ell_2}^2t\right]}\,d\xi.
	\end{equation}
	Recall \Cref{nota}.
	Substitute $\xi_1=R+\lambda R^{1/2}$ (so $\abs\lambda\le1$ on $\operatorname{supp}\widehat\phi$, and $d\xi_1=R^{1/2}\,d\lambda$, which exactly cancels the prefactor $R^{-1/2}$), and inside each summand $\xi'=D\ell_1+\eta'$ with $\abs{\eta'}\le1$. Expanding
	\[
	\xi_1^2=R^2+2\lambda R^{3/2}+\lambda^2R,\qquad \abs{\xi'}^2=\abs{D\ell_1}^2+2D\ell_1\cdot\eta'+\abs{\eta'}^2,
	\]
	and regrouping the phase, \eqref{eq:prop-1} becomes
	\begin{equation}\label{eq:prop-2}
		e^{it\Delta}F(x,y)=e^{2\pi i(Rx_1+R^2t)}\,I_\lambda(x_1,t)\sum_{\ell\in\Lambda}e^{2\pi i\left(Dx'\cdot\ell_1+Dy\cdot\ell_2+D^2t\abs\ell^2\right)} \left(\prod_{i=2}^n I_{\eta_i}(x_i,t;\ell_{1,i})\right).
	\end{equation}
	Here $\abs\ell^2:=\abs{\ell_1}^2+\abs{\ell_2}^2,$
	\begin{equation}\label{eq:Lambda}
		I_\lambda(x_1,t)=\int_{-1}^{1}\widehat\phi(\lambda)\,e^{2\pi i\left(\lambda R^{1/2}(x_1+2Rt)+\lambda^2Rt\right)}\,d\lambda,
	\end{equation}
	and $I_{\eta_i}$ are the analogous one-dimensional integrals arising from the tensor factorization of $\widehat\Phi$, with phase $\eta_{i} x_i +2D\ell_{1,i}\eta_it+\eta_i^2t$. Specifically, 
	\begin{equation*}
		I_{\eta_i}(x_i,t;\ell_{1,i})=\int_{-1}^1\widehat\varphi(\eta_i)\,e^{2\pi i\left(\eta_i(x_i+2D\ell_{1,i}t)+\eta_i^2t\right)}\,d\eta_i.
	\end{equation*}
	
	\begin{itemize}
		\item[-] \underline{\textbf{\textit{Phase-removal argument for} $I_{\lambda}$}}.
	\end{itemize}
	Applying \Cref{lem:vdc} to \eqref{eq:Lambda} with $\mu=\widehat\phi$, $[a,b]=[-1,1]$, $\Psi(\lambda)=\lambda R^{1/2}(x_1+2Rt)+\lambda^2Rt$, so that $\Psi'(\lambda)=R^{1/2}(x_1+2Rt)+2\lambda Rt$ and 
	\[
	\norm{\Psi'}_{L^\infty[-1,1]}\le R^{1/2}\abs{x_1+2Rt}+2R\abs t.
	\]
	Since $\int\widehat\phi=\phi(0)=1$, \Cref{lem:vdc} gives
	\[
	\abs{I_\lambda(x_1,t)}\ \ge\ 1-4\pi\norm{\widehat\phi}_{L^1}\Big(R^{1/2}\abs{x_1+2Rt}+2R\abs t\Big).
	\]
	Hence, if $c_0:=\big(24\pi\norm{\widehat\phi}_{L^1}\big)^{-1}$ and
	\begin{equation}\label{eq:localization}
		R^{1/2}\abs{x_1+2Rt}\le c_0,\qquad R\abs t\le c_0,
	\end{equation}
	then \begin{equation}\label{eq:Ilambda-final}
		\abs{I_\lambda(x_1,t)}\ge\tfrac12.
	\end{equation}
	\begin{itemize}
		\item[-] \underline{\textbf{\textit{Reduction to arithmetic sum}}}. 
	\end{itemize}
	Put  $$S:=\sum_{\ell_1\in\Lambda_{1}}e^{2\pi i\left(Dx'\cdot\ell_1+D^2t\abs{\ell_{1}}^2\right)} \left(\prod_{i=2}^n I_{\eta_i}(x_i,t;\ell_{1,i})\right).$$
	We should think of $S$ corresponding to the Euclidean component in \eqref{eq:prop-2}.
	Note that 
	\begin{align*}
		\sum_{\ell \in\Lambda}e^{2\pi i\left(Dx'\cdot\ell +  Dy \cdot \ell_2+D^2t\abs{\ell}^2\right)} \left(\prod_{i=2}^n I_{\eta_i}(x_i,t;\ell_{1,i})\right) = \sum_{\ell_2\in\Lambda_2}e^{2\pi i\left(Dy\cdot\ell_2+D^2t\abs{\ell_{2}}^2\right)} \cdot  S
	\end{align*}
	We fix $q\ge1$ and integers $p_1,\dots,p_n$ with $\gcd(p_1,q)=1$, and set
	\begin{equation}\label{eq:rational-approx euclidean}
		D^2t=\frac{p_1}q,\qquad Dx_i=\frac{p_i}q+\varepsilon_i\ (i=2,\dots,n),
	\end{equation} with $\abs{\varepsilon_i} \lesssim\ D/R.$
	Writing $\tilde{p}=(p_2,\dots,p_n)$ and $\tilde{\varepsilon}=(\varepsilon_2,\dots,\varepsilon_n),$
	\begin{equation}\label{eq:phase-expansion euclidean}
		Dx'\cdot\ell_1+D^2t\abs{\ell_1}^2=\frac{\tilde{p}\cdot\ell_1+p_1\abs{\ell_1}^2}{q}+\tilde{\varepsilon} \cdot\ell_1.
	\end{equation}
	Under assumption \eqref{eq:rational-approx euclidean}, by \cite[page 8, Eq. (3.7)]{Eceizabarrena2025BourgainCounter} (see also \cite[Section 3.1]{Eceizabarrena2022pointwiseFractals}, \cite{Pierce2020BourgainCounter}), we  can reduce $S$ to the arithmetic sum. Specifically, we have 
	\begin{equation}\label{euclidean sum}
		|S|  \sim \left| \sum_{{\ell_{1}}\in\Lambda_1}e^{2\pi i\left(\tilde{p}\cdot{\ell_{1}}+p_1\abs{\ell_{1}}^2\right)/q} e^{2 \pi i \tilde{\varepsilon} \cdot {\ell_{1}}} \right|
	\end{equation}
	In order to handle the contribution from torus component in \eqref{eq:prop-2}, we  further assume, integers $p_{n+1},\dots,p_{n+m}$, and set
	\begin{equation}\label{eq:rational-approx torus}
		Dy_j=\frac{p_{n+j}}q+\varepsilon_{n+j}\ (j=1,\dots,m),
	\end{equation}
	with $\abs{\varepsilon_{n+j}} \lesssim\ D/R.$
	Writing $\tilde{p}'=(p_{n+1},\dots,p_{n+m})$ and $\tilde{\varepsilon}'=(\varepsilon_{n+1},\dots,\varepsilon_{n+m}),$
	\begin{equation}\label{eq:phase-expansion torus}
		Dy\cdot\ell_2+D^2t\abs{\ell_{2}}^2=\frac{\tilde{p}'\cdot\ell_{2}+p_1\abs{\ell_{2}}^2}{q}+\tilde{\varepsilon}'\cdot\ell_2.
	\end{equation}
	Thus we obtain
	\begin{equation}\label{torus sum}
		\left|\sum_{\ell_2\in\Lambda_2}e^{2\pi i\left(Dy\cdot\ell_2+D^2t\abs{\ell_{2}}^2\right)}\right|= \left| \sum_{{\ell_{2}}\in\Lambda_2}e^{2\pi i\left(\tilde{p}'\cdot{\ell_{2}}+p_1\abs{\ell_{2}}^2\right)/q} e^{2 \pi i \tilde{\varepsilon}' \cdot {\ell_{2}}} \right|
	\end{equation}
	Put 
	\[p'=(\tilde{p},\tilde{p}')=(p_2,\dots,p_n,p_{n+1},\dots,p_{n+m}), \varepsilon=(\tilde{\varepsilon},\tilde{\varepsilon}')=(\varepsilon_2,\dots,\varepsilon_n,\varepsilon_{n+1},\dots,\varepsilon_{n+m}).\]
	Then
	\begin{equation}\label{eq:phase-expansion}
		Dx'\cdot\ell_1+Dy\cdot\ell_2+D^2t\abs\ell^2=\frac{p'\cdot\ell+p_1\abs\ell^2}{q}+\varepsilon\cdot\ell,
	\end{equation}
	with 
	\begin{equation}\label{eq:eps-delta}
		\abs{\varepsilon_i} \lesssim\ D/R.
	\end{equation}
	Under this identification, by \eqref{euclidean sum} and \eqref{torus sum},
	\begin{equation*}
		\left| \sum_{\ell_2\in\Lambda_2}e^{2\pi i\left(Dy\cdot\ell_2+D^2t\abs{\ell_{2}}^2\right)} \cdot  S \right| \sim \left| \sum_{\ell\in\Lambda}e^{2\pi i\left(p'\cdot\ell+p_1\abs\ell^2\right)/q} e^{2 \pi i \varepsilon \cdot \ell} \right|.
	\end{equation*}
	Conditions \eqref{eq:localization},\eqref{eq:rational-approx euclidean},\eqref{eq:rational-approx torus} and \eqref{eq:eps-delta}    allows us to estimate the right side of the above expression by a Gauss sum (see \cite[page 8, Eq. (3.8)]{Eceizabarrena2025BourgainCounter} (see also \cite[Section 3.1]{Eceizabarrena2022pointwiseFractals}, \cite{Pierce2020BourgainCounter}), i.e.,
	\begin{equation}\label{Gauss sum for waveguide}
		\left| \sum_{\ell\in\Lambda}e^{2\pi i\left(p'\cdot\ell+p_1\abs\ell^2\right)/q} e^{2 \pi i \varepsilon \cdot \ell} \right| \sim \prod_{j=2}^{n+m} \left(\frac{R}{Dq} \right)  \left| \sum_{v_j=1}^{q}  e^{2\pi i(p_j v_j+p_1{v_j}^2)/q}\right|.
	\end{equation}
	Now for suitable choice of $q,p_1,p'$ (for instance, choosing $q$ odd or, alternatively, $q$ even and $p_j \equiv q/2 \pmod{2}$), each Gauss sum has size $\sqrt{q}.$ So
	using \eqref{eq:Ilambda-final} and \eqref{Gauss sum for waveguide} to \eqref{eq:prop-2}, we get
	\begin{equation}\label{eq:S-estimate}
		\abs{e^{it\Delta}F(x,y)} \ \sim\ \left(\frac{R}{D\sqrt q}\right)^{n+m-1}.
	\end{equation}
	Thus by \eqref{eq:norm-f}, we get
	\begin{equation}\label{eq:pointwise-final}
		\frac{\abs{e^{it\Delta}F(x,y)}}{\norm{F}_{L^2}}\ \gtrsim\ R^{1/4}\left(\frac{R}{Dq}\right)^{\frac{n+m-1}{2}}.
	\end{equation}
	\begin{itemize}
		\item[--] \textbf{Step II: The resonant set and its Lebesgue measure}
	\end{itemize}
	We point out that in previous step the bound \eqref{eq:pointwise-final} does not hold for every $(x,y)$. It only holds on certain spacetime boxes, coming from the rational approximations \eqref{eq:rational-approx euclidean} and \eqref{eq:rational-approx torus}, with denominator $q$ near a parameter $Q=Q(R)$, which we still need to choose.
	Let $X_R$ be the union of all these boxes, one for each admissible $q\sim Q$ and each choice of residues $(p_1,p')$ (to be defined more specifically below in \eqref{dfxr}). By construction, \eqref{eq:pointwise-final} holds for every $(x,y)\in X_R$.
	
	If $\abs{X_R}\to0$ as $R\to\infty$, the pointwise bound \eqref{eq:pointwise-final} would hold only on a vanishing set. We therefore prove
	$$\abs{X_R}\gtrsim1 \quad \text{uniformly in} \ R.$$
	Let us discuss the conditions imposed in \eqref{eq:localization}, \eqref{eq:rational-approx euclidean}, \eqref{eq:rational-approx torus} and \eqref{eq:eps-delta}. First, the second conditions in \eqref{eq:rational-approx euclidean}, \eqref{eq:rational-approx torus} and \eqref{eq:eps-delta} can be written as
	
	$$z \in B^{n+m-1}\left(\frac{p'}{Dq},\frac{c}{R}\right),$$
	where $z=(x',y)\in\mathbb R^{n-1}\times\mathbb T^m.$
	On the other hand, the first conditions in \eqref{eq:localization} and \eqref{eq:rational-approx euclidean}  give
	$$x_1 \in B^1\left(\frac{Rp_1}{D^2q},\frac{c}{R^{1/2}}\right).$$
	Taking this  consideration into account, we define
	\begin{equation}\label{dfxr}
		X_R := \bigcup_{q\sim Q}\bigcup_{(p_1,p')} B^1\!\left(\frac{Rp_1}{D^2q},\,\frac{c}{R^{1/2}}\right)\times B^{n+m-1}\!\left(\frac{p'}{Dq},\,\frac{c}{R}\right)\subset[0,1]^{n+m}.   
	\end{equation}
	Write $z=(x',y)\in\mathbb R^{n-1}\times\mathbb T^m$, and define $$T(x_1,z)=\left(\frac{D^2}{R}x_1,\,Dz\right).$$
	For the Euclidean coordinates $x_1,x'$ this is an ordinary linear change of variables, contributing a factor $\frac{D^2}{R}\cdot D^{n-1}$ to the Jacobian. For the torus coordinate $y$, since $D$ is a positive integer the dilation $y\mapsto Dy$ is compatible with the $1$-periodicity of $y,$ locally it contributes a factor $D^m$ to the Jacobian, and the image of the fundamental domain $[0,1)^m$ consists of $D^m$ periodic copies of itself. Combining these, we have  $$|T(E)| = \frac{D^2}{R}D^{n+m-1}\,|E|$$
	for every measurable set $E$, where $|E|$ is the Lebesgue measure, understood on the lifted domain with the periodic identification in $y$ just described.
	
	Applying $T$ to each ball of $X_R$ sends center $Rp_1/(D^2q)\mapsto p_1/q$, with radius $c/R^{1/2}\mapsto cD^2/R^{3/2}$, and center $p'/(Dq)\mapsto p'/q$, with radius $c/R\mapsto cD/R$. As $X_R\subset[0,1]^{n+m}$, the parameters $p_1,p'$ range over $[0,q)$, and incrementing either by $q$ shifts the corresponding center $p_1/q$ or $p'/q$ by an integer vector. These shifts account exactly for the $\frac{D^2}{R}D^{n+m-1}$ periodic copies produced by $T$. So $T(X_R)$ consists of exactly $\frac{D^2}{R}D^{n+m-1}$ integer translates of
	$$\widetilde X_R := \bigcup_{q\sim Q}\bigcup_{(p_1,p')\in[0,q)^{n+m}} B^1\!\left(\frac{p_1}{q},\,\frac{cD^2}{R^{3/2}}\right)\times B^{n+m-1}\!\left(\frac{p'}{q},\,\frac{cD}{R}\right)\subset[0,1)^n \times \mathbb T^{m},$$
	so $|T(X_R)| = \frac{D^2}{R}D^{n+m-1}\,|\widetilde X_R|$. Comparing with the Jacobian relation above,
	$$|X_R| = |\widetilde X_R|.$$
	Set
	$$h_1(Q) = \frac{D^2}{R^{3/2}}, \qquad h_2(Q) = \frac{D}{R}.$$
	With
	$$Q = R^{\frac{n+m-1}{2(n+m+1)}}, \qquad D = R^{\frac{n+m+2}{2(n+m+1)}},$$
	we get $(QD)^{n+m+1}\sim R^{3/2}R^{n+m-1}$, i.e. $Q^{n+m+1}h_1(Q)h_2(Q)^{n+m-1}\sim 1$. By \Cref{lem:A.1},
	$$|\widetilde X_R|\sim 1,$$
	hence $|X_R|\sim 1$. Thus, for $z=(x,y)\in X_R$,
	$$\frac{|e^{it\Delta}F(x,y)|}{\|F\|_{L^2(\mathbb R^n\times\mathbb T^m)}}\gtrsim R^{\frac{n+m}{2(n+m+1)}}.$$
	Setting $F_R:=F/\norm F_{L^2}$ (see \eqref{eq:def-F}), and using $\abs{X_R}\gtrsim1$ uniformly in $R$,
	\[
	R^{-s}\norm{\sup_{0<t<1}\abs{e^{it\Delta}F_R}}_{L^1(B^n(0,1)\times\T^m)}\ \gtrsim\ R^{\frac{n+m}{2(n+m+1)}-s}\ \longrightarrow\ \infty
	\]
	for every $s<\dfrac{n+m}{2(n+m+1)}$.
	
	It remains to  choose $R_k$ and $F_k$ explicitly. The construction of $F$ in \eqref{eq:def-F} needs $D=D(R)=R^{(n+m+2)/(2(n+m+1))}$ to satisfy $D\ell_2\in\Z^m$ for every $\ell_2\in\Lambda_2$; an integer value of $D$ suffices. Setting
	\begin{equation}\label{choosing R_k}
		R_k:=k^{2(n+m+1)}, \qquad D(R_k)=k^{n+m+2} \in\Z,\qquad k\ge k_0(n,m,\phi). 
	\end{equation}
	Put $F_k:=F_{R_k}/\norm{F_{R_k}}_{L^2}$. Then $R_k\to\infty$, $\norm{F_k}_{L^2}=1$, and $\widehat{F_k}$ is supported in $\{\abs{(\xi,k)}\sim R_k\}$ by \eqref{eq:supp-f}; the limit above with $R=R_k$ is \Cref{Necessary condition for proposition}. This completes the proof of \Cref{Necessary condition for proposition}.
\end{proofof}

\begin{proposition}[point-wise convergence implies maximal estimate]\label{point-wise convergence implies maximal estimate}
	Suppose  for some $s > 0$, the pointwise convergence
	\[
	\lim_{t \to 0} e^{it\Delta} f(x,y) = f(x,y) \quad  \text{holds  for a.e.} \  (x,y), \forall f \in H^s(\mathbb{R}^n \times \mathbb{T}^m).
	\]
	Then the maximal estimate
	\begin{equation}\label{eq:Convg for Hs}
		\left\| \sup_{0 < t \leq 1} \left| e^{it\Delta} f(x,y) \right| \right\|_{L^1_{x,y}\left(B^n(0,1) \times \mathbb{T}^m\right)}
		\lesssim_s
		\|f\|_{H^s(\mathbb{R}^n \times \mathbb{T}^m)}
	\end{equation}
	holds for all $f \in H^s(\mathbb{R}^n \times \mathbb{T}^m)$.
\end{proposition}
\begin{proof}
	To prove the maximal estimate, we first introduce the operator $G_s$ defined by
	\[
	\widehat{G_sg}(\xi,k)
	=
	(1+\abs{\xi}^2+\abs{k}^2)^{-s/2}\widehat{g}(\xi,k), \quad (\xi, k) \in \R^n \times \Z^m.
	\]
	By the definition of the Sobolev space on $\R^n\times\T^m$, the operator
	\[
	G_s:L^2(\R^n\times\T^m)\rightarrow H^s(\R^n\times\T^m)
	\]
	is an isometric isomorphism.  Consequently, $g\in L^2(\R^n\times\T^m)$ if and only if $G_sg\in H^s(\R^n\times\T^m).$
	Hence, the hypothesis of the proposition is equivalent to
	\begin{equation}\label{eq:convG}
		\lim_{t\to0}e^{it\Delta}G_sg(x,y)
		=
		G_sg(x,y),
		\qquad
		\text{for a.e. }(x,y)\in\R^n\times\T^m,
	\end{equation}
	for every $g\in L^2(\R^n\times\T^m)$.
	
	We divide the proof into two steps.

	\begin{itemize}
		\item[-] \textbf{Step 1. Weak-type estimate}
	\end{itemize}

	Let $\{t_k\}_{k=1}^{\infty}$ be an enumeration of the positive rational numbers contained in $(0,1]$, and define
	\[
	T^{(*)}g(x,y)
	=
	\sup_{k\ge1}
	\abs{e^{it_k\Delta}g(x,y)}.
	\]
	Since, for every fixed $(x,y)$, the map
	\[
	t\longmapsto e^{it\Delta}g(x,y)
	\]
	is continuous, it follows from \cite[\S12]{Stein1961limitsoperators} that
	\begin{equation}\label{maximal continuous=discrete}
		\sup_{0<t\le1}\abs{e^{it\Delta}g(x,y)}=
		\sup_{k\ge1}\abs{e^{it_k\Delta}g(x,y)}.
	\end{equation}
	Therefore, it is enough to establish the estimate for $T^{(*)}.$ Fix $s>s_c,$ for each $k$, the operator $e^{it_k\Delta}G_s$ is bounded on $L^2(\R^n\times\T^m),$
	since $e^{it\Delta}$ is a unitary on $L^2(\R^n\times\T^m)$ and $G_s$ is bounded on $L^2(\R^n\times\T^m).$ Hence, if $g_j\rightarrow g $ in $L^2(\R^n\times\T^m),$
	we also have
	\[
	e^{it_k\Delta}G_sg_j
	\longrightarrow
	e^{it_k\Delta}G_sg
	\]
	in measure. Now by  assumption \eqref{eq:convG},
	\[
	\lim_{k \to \infty} e^{it_k\Delta}G_sg(x,y)=G_sg(x,y)
	\]
	for almost every $(x,y)\in\R^n\times\T^m$. This implies, $T^{(*)}G_sg(x,y)<\infty$ on a set of positive measure.
	
	Thus by weak-type maximal principle (see \cite[Chapter X, \S3.4]{Stein1993orthogonality}), for the compact set $X=B^n(0,1)\times\T^m,$ there exists a constant $C_s>0$ such that for every $\alpha>0$,
	\begin{equation}\label{eq:weaktype for discrete}
		\abs{
			\left\{
			(x,y)\in X:
			T^{(*)}G_sg(x,y)>\alpha
			\right\}
		}
		\le
		\frac{C_s}{\alpha^2}
		\norm{g}_{L^2(\R^n\times\T^m)}^2.
	\end{equation}
	By \eqref{maximal continuous=discrete} and \eqref{eq:weaktype for discrete},
	\begin{equation}\label{eq:weaktype for continuous}
		\left\|\sup_{0<t\le1}\abs{e^{it\Delta}G_s g(x,y)}  \right\|_{L^{2,\infty}(X)}
		\le C_s
		\norm{g}_{L^2(\R^n\times\T^m)}.
	\end{equation}

	\begin{itemize}
		\item[-] \textbf{Step 2. Lorentz space estimate}:
	\end{itemize}
	
	Since $X$ has finite measure, it is enough to prove that
	\[
	\norm{F}_{L^1(X)}
	\le
	2\abs{X}^{1/2}
	\norm{F}_{L^{2,\infty}(X)}
	\]
	for every $F\in L^{2,\infty}(X)$ (for definition see \eqref{weak Lp space definition}) and here $|X|$ denotes measure of the set $X$.
	
	Let
	\[
	\lambda(t)
	=
	\abs{
		\left\{
		(x,y)\in X:
		\abs{F(x,y)}>t
		\right\}
	}
	\]
	be the distribution function of $F$. By the layer-cake representation,
	\[
	\norm{F}_{L^1(X)}
	=
	\int_0^\infty\lambda(t)\,dt.
	\]
	Fix $A>0$. Splitting the integral gives
	\[
	\norm{F}_{L^1(X)}
	=
	\int_0^A\lambda(t)\,dt
	+
	\int_A^\infty\lambda(t)\,dt.
	\]
	
	Since $\lambda(t)\le \abs{X}$, we have
	\[
	\int_0^A\lambda(t)\,dt
	\le
	A\abs{X}.
	\]
	
	On the other hand, we have
	\[
	\int_A^\infty\lambda(t)\,dt= \int_A^\infty t^2\lambda(t)\frac{dt}{t^2}
	\le
	\norm{F}_{L^{2,\infty}(X)}^2
	\int_A^\infty\frac{dt}{t^2}
	=
	\frac{\norm{F}_{L^{2,\infty}(X)}^2}{A}.
	\]
	
	Combining the above estimates and choosing $A=
	\abs{X}^{-1/2}
	\norm{F}_{L^{2,\infty}(X)}$, we get
	\begin{equation}\label{Holder for weak space}
		\norm{F}_{L^1(X)}\le
		2\abs{X}^{1/2}
		\norm{F}_{L^{2,\infty}(X)}.
	\end{equation}
	
	Since $G_sg=f$ and by \eqref{eq:weaktype for continuous} and \eqref{Holder for weak space}, we conclude that
	\[
	\norm{
		\sup_{0<t\le1}
		\abs{e^{it\Delta}f}
	}_{L^1(B^n(0,1)\times\T^m)}
	\lesssim_s
	\norm{f}_{H^s(\R^n\times\T^m)},
	\]
	which is the desired maximal estimate.
\end{proof}

\begin{proofof}{Necessary condition for theorem}
	We prove this result using contradiction method, if possible for some $s_0 < \frac{n+m}{2(n+m+1)}$, assume
	\[
	\lim_{t \to 0} e^{it\Delta} f(x,y) = f(x,y) \quad \text{for a.e. } (x,y) \in \mathbb{R}^n \times \mathbb{T}^m
	\]
	holds for every $f \in H^{s_0}(\mathbb{R}^n \times \mathbb{T}^m)$.
	
	By \Cref{Necessary condition for proposition}, there exist a sequence $R_k \to \infty$ and functions $F_k \in L^2(\mathbb{R}^n \times \mathbb{T}^m)$ with $\|F_k\|_{L^2} = 1$ and Fourier support contained in $\{|\xi|^2 + |\ell|^2 \sim R_k^2\}$ such that
	\[
	R_k^{-s_0} \left\| \sup_{0 < t < 1} |e^{it\Delta} F_k| \right\|_{L^1(B(0,1) \times \mathbb{T}^m)} \to \infty.
	\]
	Since $\operatorname{supp}\widehat{F_k} \subset \{|(\xi,\ell)| \sim R_k\}$, we have $\|F_k\|_{H^{s_0}(\mathbb{R}^n \times \mathbb{T}^m)} \sim R_k^{s_0} \|F_k\|_{L^2} = R_k^{s_0}$. It follows that
	\[
	\frac{\left\| \sup_{0 < t < 1} |e^{it\Delta} F_k| \right\|_{L^1(B(0,1) \times \mathbb{T}^m)}}{\|F_k\|_{H^{s_0}(\mathbb{R}^n \times \mathbb{T}^m)}} \sim R_k^{-s_0} \left\| \sup_{0 < t < 1} |e^{it\Delta} F_k| \right\|_{L^1(B(0,1) \times \mathbb{T}^m)} \to \infty.
	\]
	However, the assumed almost everywhere convergence implies the maximal estimate
	\[
	\left\| \sup_{0 < t \le 1} |e^{it\Delta} f| \right\|_{L^1(B^n(0,1) \times \mathbb{T}^m)} \lesssim_{s_0} \|f\|_{H^{s_0}(\mathbb{R}^n \times \mathbb{T}^m)}
	\]
	for all $f \in H^{s_0}(\mathbb{R}^n \times \mathbb{T}^m)$ by \eqref{eq:Convg for Hs}. This contradicts the divergence along $F_k$, completing the proof.
\end{proofof}

\section{Sufficient condition for infinite system}\label{sec-suff-inf}
In this section we shall prove \Cref{t:maximal ose T,t:maximal ose wave}.
\begin{remark}[proof strategy]\label{psfs}
	The duality principle \Cref{PL1} converts the orthonormal maximal estimate into a Hilbert--Schmidt bound on $W_1\mathcal E_N\mathcal E_N^*W_2$, computed as the $L^2$-norm of the kernel $K_N$ (\Cref{prop:strichartz for each PN T,prop:strichartz for each PN wave}). The proof of \Cref{t:maximal ose T} strongly relies on the periodic structure of the torus. In the torus case, solutions to the free Schr\"odinger equation are periodic in both time and spatial variables, which allows us to apply Young's convolution inequality together with \Cref{max stri for torus} directly. In contrast, for the waveguide case considered in \Cref{t:maximal ose wave}, solutions to the free Schr\"odinger equation are no longer periodic in time and spatial variables. Consequently, Young's convolution inequality and \Cref{max stri for wave} cannot be applied directly. To overcome this difficulty in the waveguide setting, we exploit the finite covering property of balls centered at the origin, particularly that a ball of radius two can be covered by finitely many unit balls i.e., ball of radius one.
\end{remark}
\begin{proposition}\label{prop:strichartz for each PN T}
	Let $s \geq 0,$ and let $N\in2^{\mathbb N}$ be each dyadic frequency. Then, for every $\varepsilon>0,$
	\begin{equation}\label{eq:strichartz for each PN}
		\Big\|
		\sum_j \lambda_j \bigl|e^{it\Delta}\langle\nabla\rangle^{-s} P_N f_j\bigr|^2
		\Big\|_{L_x^2L_t^\infty(\T^d \times \T)}
		\lesssim_{\varepsilon}
		N^{\frac{d}{2}+\frac{d}{d+2}-2s+\varepsilon}\|\lambda\|_{\ell^2}.
	\end{equation}
	Here $P_N$ is defined in \eqref{lpope} and $(f_j)_j$ is an orthonormal system in $L^2(\T^d),$ and $\lambda=(\lambda_j)_j\in\ell^2$.
\end{proposition}
\begin{proposition}\label{prop:strichartz for each PN wave}
	Let $s \geq 0,$ and let $N\in2^{\mathbb N}$ be each dyadic frequency. Then, for every $\varepsilon>0,$
	\begin{equation}\label{eq:strichartz for each PN wave}
		\Big\|
		\sum_j \lambda_j \bigl|e^{it\Delta}\langle\nabla\rangle^{-s} P_N f_j\bigr|^2
		\Big\|_{L_x^2L_t^\infty(B^n(0,1) \times \T^m \times (0,1])}
		\lesssim_{\varepsilon}
		N^{\frac{d}{2}+\frac{d}{d+2}-2s+\varepsilon}\|\lambda\|_{\ell^2},
	\end{equation}
	Here $P_N$ is defined in \eqref{lpope} and $(f_j)_j$ is an orthonormal system in $L^2(\wg),$ and $\lambda=(\lambda_j)_j\in\ell^2$.
\end{proposition}

In order to prove these propositions, we briefly   introduce some terminology.   
Recall $X \times I$ from \eqref{XxI}. Let $\psi_N$ be the function mentioned earlier in \eqref{psi N}.
Let $\cH=L_{z,t}^2(\M \times I),$ for $a  \in L_{\xi}^2(\hM),$ the  \textbf{Fourier extension operator} $\mathcal{E}_N$  is given by 
\begin{equation}\label{FEP}
	\mathcal{E}_N a (z,t) = \int_{\hM} a(\xi)  e^{2\pi i (z \cdot \xi + t|\xi|^2)} \psi_N(\xi)\langle \xi \rangle^{-s} d\xi, \quad(z, t) \in X \times I. 
\end{equation}
By Plancherel's theorem,  it follows that   $\mathcal{E}_{N}:L_{\xi}^2(\hM) \to \cH $ is a bounded operator. In fact, we have 
\begin{align*} 
	\|\mathcal{E}_{N}a\|_{\cH}^{2}&= \int_{I} \|\mathcal{E}_{N}a(\cdot,t)\|^2_{L_{z}^{2}(\M)} \, dt =\int_{I} \|\widehat{\mathcal{E}_{N}a(\cdot,t)}(\xi)\|^2_{L_{\xi}^{2}(\hM)} \, dt \lesssim \|a\|_{L_{\xi}^{2}(\hM)}^{2}.
\end{align*}    
The dual\footnote{The dual operator of $\mathcal{E}_N$ means that
	$\langle \mathcal{E}_N a, F \rangle_{\cH} = \langle a, \mathcal{E}_N^* F \rangle_{L^2(\hM)}
	$
	holds for all $a$ and $F$} of  the operator $\mathcal{E}_N$ is denoted by $\mathcal{E}^*_N$-so called the \textbf{Fourier restriction operator}. Specifically, $\mathcal{E}_N^*: \cH \to L_{\xi}^2(\hM)$ is
given by
\begin{equation}\label{FRO}
	F \mapsto \mathcal{E}_N^* F (\xi) =
	\int_{X \times I} F(z, t) e^{-2\pi i (z \cdot \xi + t|\xi|^2)} \psi_N(\xi)\langle \xi \rangle^{-s} \, dz dt.  
\end{equation} 
We note that a composition of $\mathcal{E}_N$ and $\mathcal{E}^*_N$  gives us a convolution operator on $\cH$. In fact, we have 
\begin{flalign}\label{P1}
	\mathcal{E}_N \circ \mathcal{E}_N^* F(z,t) &= \int_{\hM } \mathcal{E}_N^*F(\xi) e^{2\pi i (z \cdot \xi + t|\xi|^2)} \psi_N(\xi)\langle \xi \rangle^{-s} \, d\xi  \nonumber \\
	&=\int_{ \hM }\int_{X \times I} F(z',t') e^{-2\pi i (z' \cdot \xi + t'|\xi|^2)}e^{2\pi i (z \cdot \xi + t|\xi|^2)} \left(\psi_N(\xi)\langle \xi \rangle^{-s}\right)^{2} \, d\xi \, dz' \, dt' \nonumber \\
	&=\int_{X \times I} F(z',t')\int_{ \hM }e^{2\pi i [(z-z') \cdot \xi + (t-t')|\xi|^2]} \left(\psi_N(\xi)\langle \xi \rangle^{-s}\right)^{2} \, d\xi \, dz' \, dt' \nonumber \\
	& =  \int_{X \times I} K_N (z - z', t - t') F(z', t') \, dz' dt'.
\end{flalign}  
where 
\begin{equation}\label{kernel of extension}
	K_N (z,t) = \int_{\hM} e^{2\pi i (z \cdot \xi + t |\xi|^2)} \left(\psi_N(\xi)\langle \xi \rangle^{-s}\right)^{2} \, d\xi.
\end{equation}

\begin{remark}\label{usrev} We can restate inequalities \eqref{eq:strichartz for each PN} and \eqref{eq:strichartz for each PN wave} as follows: it holds for every $N > 1$, any $\lambda \in \ell^{\alpha'}$, and any orthonormal system $(f_j)_j \subset L_{z}^2(\M)$ if and only if
	\begin{equation}\label{P2}
		\left\| \sum_j \lambda_j |\mathcal{E}_N a_j|^2 \right\|_{L_z^2 L_t^\infty(X \times I)} \leq  C N^{\frac{d}{2}+\frac{d}{d+2}-2s+\varepsilon}\|\lambda\|_{\ell^2}
	\end{equation}
	holds for any $N > 1$, $\lambda \in \ell^{2}$, and any ONS $(a_j)_j \subset L_{\xi}^{2}(\hM)$. This is
	because if we let $a_j = \widehat{f_j}$, then the orthonormality of $(f_j)_j$ in $L_{z}^2(\M)$ is equivalent
	to the one of $(a_j)_j$ in $L_{\xi}^{2}(\hM)$ and \[
	e^{it\Delta}\langle\nabla\rangle^{-s}P_N f_j=\mathcal E_N a_j.
	\]
\end{remark}

\begin{proofof}{prop:strichartz for each PN T}
	Using \Cref{usrev}, we need to prove \eqref{P2}. By \Cref{PL1} it is equivalent to the following:
	\begin{equation}\label{dual strichartz}
		\|W_1 \, \mathcal E_N \mathcal E_N^\ast W_2\|_{\Sp^2(L^2(\T^{d+1}))} 
		\lesssim_{\varepsilon} 
		N^{\frac{d}{2}+\frac{d}{d+2}-2s+\varepsilon}
		\|W_1\|_{L_x^4 L_t^2}
		\|W_2\|_{L_x^4 L_t^2}.
	\end{equation}
	We first find the kernel of the operator
	$W_1 \mathcal E_N \mathcal E_N^\ast W_2$. In fact, for $f \in L^{2}(X \times I)$, here $X=\T^d$ and $I=\T.$ We have
	\[
	\begin{aligned}
		W_1 \mathcal E_N \mathcal E_N^\ast W_2 f(x,t)
		&= W_1(x,t) \int_{ X \times I} K_N(x-x',t-t') W_2(x',t') f(x',t') \, dx' \, dt' \\
		&= \int_{X \times I}
		W_1(x,t) K_N(x-x',t-t') W_2(x',t') f(x',t') \, dx' \, dt' \\
		&= \int_{ X \times I} \widetilde{K}(x,t,x',t') f(x',t') \, dx' \, dt',
	\end{aligned}
	\]
	where the kernel is given by
	\[
	\widetilde{K}(x,t,x',t') = W_1(x,t) K_N(x-x',t-t') W_2(x',t').
	\]
	Since $\Sp^{2}$-norm is the Hilbert--Schmidt norm and is equal to $L^2$-norm of the kernel, so LHS of \eqref{dual strichartz}, we have
	\begin{align}\label{C2 norm LHS}
		\|W_{1}\mathcal E_N \mathcal E_N^\ast W_{2}\|_{\Sp^{2}(L^{2}(\T^{d+1}))}^{2}
		&=\|\widetilde{K}\|_{L^{2}(\T^{d+1} \times \T^{d+1})}^{2} \nonumber \\
		&=\int_{\T \times \T}\int_{\T^d \times \T^d}
		|W_{1}(x,t)K_{N}(x-x',t-t')W_{2}(x',t')|^{2}\,dx\,dt\,dx'\,dt' \nonumber \\
		&\leq \int_{\T^d \times \T^d} \|W_{1}(x, \cdot)\|^2_{L^2_t} \|K_{N}(x-x', \cdot)\|^2_{L^\infty_t}\|W_{2}(x', \cdot)\|^2_{L^2_t} \,dx\,dx' \nonumber \\
		&\leq  \|\|W_{1}\|^2_{L^2_t} \|_{L^{2}_x} \|\|K_{N}\|^2_{L^\infty_t}\|_{L^{1}_x} \|\|W_{2}\|^2_{L^2_t}\|_{L^{2}_x}   \nonumber\\
		&= \|W_{1}\|^2_{L^4_x L^2_t}  \|K_{N}\|^2_{L^2_x L^\infty_t} \|W_{2}\|^2_{L^4_x L^2_t}
	\end{align}

	where
	\[
	K_N(x,t)
	=
	\sum_{n\in\mathbb Z^d}
	e^{2\pi i(x\cdot n + t|n|^2)}
	\langle n\rangle^{-2s}\psi_N^2(n).
	\]
	Now, by \Cref{max stri for torus}, we get
	\begin{align}\label{KN estimate}
		\|K_{N}\|_{L^2_x L^\infty_t}&=\left\| \sup_{0<t\le1}\left|\sum_{n\in\mathbb Z^d}
		e^{2\pi i(x\cdot n + t|n|^2)}
		\langle n\rangle^{-2s}\psi_N^2(n) \right|\right\|_{L^2_x} \nonumber\\
		&\lesssim_{\varepsilon} 
		N^{\frac{d}{d+2}+\varepsilon} \left(\sum_{n \in \Z^d} \langle n\rangle^{-4s}\psi_N^2(n)\right)^{\frac{1}{2}}  \nonumber\\
		&=
		N^{\frac{d}{d+2}+\varepsilon} \left(\sum_{|n| \sim N} \langle n\rangle^{-4s}\psi_N^2(n)\right)^{\frac{1}{2}}\nonumber\\
		&\lesssim 
		N^{\frac{d}{2}+\frac{d}{d+2}-2s+\varepsilon}.
	\end{align}
	Using \eqref{KN estimate} in \eqref{C2 norm LHS}, we get
	$$
	\|W_1\,\mathcal E_N\mathcal E_N^\ast W_2\|_{\Sp^2(L^2(\mathbb T^{d+1}))}
	\lesssim_{\varepsilon}  N^{\frac{d}{2}+\frac{d}{d+2}-2s+\varepsilon}
	\|W_1\|_{L_x^4 L_t^2}
	\|W_2\|_{L_x^4 L_t^2}.
	$$
\end{proofof}
\begin{proofof}{prop:strichartz for each PN wave}
	Using \Cref{usrev}, we need to prove \eqref{P2}. By \Cref{PL1} it is equivalent to the following:
	\begin{equation}\label{dual strichartz wave}
		\|W_1 \, \mathcal E_N \mathcal E_N^\ast W_2\|_{\Sp^2(L^2(X \times (0,1]))}
		\lesssim_{\varepsilon}
		N^{\frac{d}{2}+\frac{d}{d+2}-2s+\varepsilon}
		\|W_1\|_{L_x^4 L_t^2}
		\|W_2\|_{L_x^4 L_t^2}.
	\end{equation}
	We first find the kernel of the operator
	$W_1 \mathcal E_N \mathcal E_N^\ast W_2$. In fact, for $f \in L^{2}(X \times (0,1])$, we have
	\[ 
	\begin{aligned}
		&W_1 \mathcal E_N \mathcal E_N^\ast W_2 f(x,y,t)
		= W_1(x,y,t) K_N * (W_2 f)(x,y,t) \\
		&= W_1(x,y,t) \int_{ X \times (0,1]} 
		K_N(x-x',y-y',t-t') W_2(x',y',t') f(x',y',t') \, dx' \, dy' \, dt' \\
		&= \int_{ X \times (0,1]}
		W_1(x,y,t)\, K_N(x-x',y-y',t-t')\, W_2 (x',y',t')\, f(x',y',t') \, dx' \, dy' \, dt' \\
		&= \int_{ X \times (0,1]} 
		\widetilde{K}(x,y,t,x',y',t')\, f(x',y',t') \, dx' \, dy' \, dt',
	\end{aligned}
	\]
	where the kernel is given by
	\[
	\widetilde{K}(x,y,t,x',y',t') 
	= W_1(x,y,t)\, K_N(x-x',y-y',t-t')\, W_2(x',y',t').
	\]
	Since $\Sp^{2}$-norm is the Hilbert--Schmidt norm and is equal to $L^2$-norm of the kernel,
	by using the fact that 
	\[
	\mathbf{1}_{\{|x|\le1\}} \cdot \mathbf{1}_{\{|x'|\le1\}} \cdot \mathbf{1}_{\{|x-x'|\le2\}}
	= \mathbf{1}_{\{ |x|\le1,\ |x'|\le1.\}}, \qquad \text{ for } x,x' \in \R^n ,
	\] and Young's convolution inequality  
	\begin{align}\label{C2 norm LHS wave}
		&\|W_{1}\mathcal E_N \mathcal E_N^\ast W_{2}\|_{\Sp^{2}(L^{2}(I \times B^n(0,1) \times \T^m))}^{2}
		=\|\widetilde{K}\|_{L^{2}((I \times B^n(0,1) \times \T^m)^2)}^{2} \nonumber \\
		&=\int_{(0,1]^2}
		\int_{X \times X}
		|W_{1}(x,y,t)\,K_{N}(x-x',y-y',t-t')\,W_{2}(x',y',t')|^{2} dx\,dx'\,dy\,dy'\,dt\,dt' \nonumber \\
		&\le \int_{X \times X}
		\|W_{1}(x,y,\cdot)\|^2_{L^2_t((0,1])}
		\|K_{N}(x-x',y-y',\cdot)\|^2_{L^\infty_t((-1,1))} \nonumber \\
		& \qquad \qquad \times \|W_{2}(x',y',\cdot)\|^2_{L^2_t((0,1])}
		\,dx\,dx'\,dy\,dy' \nonumber \\
		&\leq
		\int_{B^n(0,1)\times B^n(0,1)}
		\Big\|
		\|W_{1}(x,\cdot,\cdot)\|^2_{L^2_t}
		\Big\|_{L^{2}_y(\T^m)} \Big\|
		\|K_{N}(x-x',\cdot,\cdot)\|^2_{L^\infty_t((-1,1))}
		\Big\|_{L^{1}_y(\T^m)}
		\nonumber\\
		&\qquad\qquad \times
		\Big\|
		\|W_{2}\|^2_{L^2_t}
		\Big\|_{L^{2}_y(\T^m)}
		\, dx \, dx' \nonumber \\
		&= \|W_{1}\|^2_{L^4_{x,y} L^2_t(X \times (0,1])}
		\|K_{N}\|^2_{L^2_{x,y} L^\infty_t((-1,1) \times B^n(0,2) \times \T^m)}
		\|W_{2}\|^2_{L^4_{x,y} L^2_t(X \times (0,1])}.
	\end{align}
	where
	\[
	K_N(x,y,t)
	=
	\sum_{k\in\mathbb Z^m}
	\int_{\mathbb R^n}
	e^{2\pi i (x\cdot \xi + y\cdot k + t(|\xi|^2+|k|^2))}
	\langle (\xi,k)\rangle^{-2s}
	\psi_N^2(\xi,k)
	\, d\xi.
	\]
	Note that 
	\begin{align*}
		\|K_{N}\|_{L^2_{x,y} L^\infty_t((-1,1) \times B^n(0,2) \times \T^m)} &\leq \|K_{N}\|_{L^2_{x,y} L^\infty_t((-1,0] \times B^n(0,2) \times \T^m)} +\|K_{N}\|_{L^2_{x,y} L^\infty_t((0,1] \times B^n(0,2) \times \T^m)} \\
		&=:I+ II
	\end{align*}
	Let $\{B^n(x,1): x \in \overline{B^n(0,2)}\},$ since $\overline{B^n(0,2)}$ is compact so there finitely many 
	$x_i$'s in $\overline{B^n(0,2)}$ such $B^n(0,2) \subset \bigcup_{i} B^n(x_i,1).$ In order to invoke \Cref{max stri for wave}, we perform 
	change of variable $t=-1+t'$ and $x=x_i+x'$ with $x' \in B^n(0,1)$, and consequently we obtain 
	\begin{align*}
		&I^2=\|K_{N}\|^2_{L^2_{x,y} L^\infty_{t}((-1,0] \times B^n(0,2) \times \T^m)} \\
		& \leq \sum_{i} \|K_{N}\|^2_{L^2_{x,y} L^\infty_{t}((-1,0] \times B^n(x_i,1) \times \T^m)} \\
		&=\sum_{i} \|K_{N}(x_i+x',y,-1+t')\|^2_{L^2_{x',y} L^\infty_{t'}((0,1] \times B^n(0,1) \times \T^m)} \\
		&= \sum_i
		\Bigg\|
		\sup_{0<t'\le1}
		\Bigg|
		\sum_{k\in\mathbb Z^m}
		\int_{\R^n}
		e^{2\pi i (x'\cdot \xi + y\cdot k + t'(|\xi|^2+|k|^2))} \\
		&\qquad\qquad\qquad \times
		\langle (\xi,k)\rangle^{-2s}
		\psi_N^2(\xi,k)
		e^{2\pi i (x_i\cdot \xi-(|\xi|^2+|k|^2))}
		\, d\xi
		\Bigg|
		\Bigg\|^2_{L^2_{x',y}(B^n(0,1)\times\T^m)}  \\
		&\lesssim_{\varepsilon} \sum_{i} N^{\frac{2d}{d+2}+\varepsilon} \left(\sum_{k \in \Z^m} \int_{\R^n} \langle (\xi,k)\rangle^{-4s}\psi_N^2(\xi,k) \, d\xi \right) \\
		&\lesssim_{\varepsilon} \sum_{i} N^{d+\frac{2d}{d+2}-4s +\varepsilon} \\
		&\lesssim_{\varepsilon,d}  N^{d+\frac{2d}{d+2}-4s +\varepsilon}.
	\end{align*}
Similarly, we get 
	\begin{equation*}
		II \lesssim_{\varepsilon,d}  N^{\frac{d}{2}+\frac{d}{d+2}-2s +\varepsilon}
	\end{equation*}
	\begin{equation}
		\label{KN estimate wave}
		\|K_{N}\|_{L^2_{x,y} L^\infty_t((-1,1) \times B^n(0,2) \times \T^m)} \lesssim_{\varepsilon,d}  N^{\frac{d}{2}+\frac{d}{d+2}-2s +\varepsilon}
	\end{equation}
Using \eqref{KN estimate wave} in \eqref{C2 norm LHS wave}, we get
	$$
	\|W_1\,\mathcal E_N\mathcal E_N^\ast W_2\|_{\Sp^2(L^2(X \times (0,1]))}
	\lesssim_{\varepsilon}
	N^{\frac{d}{2}+\frac{d}{d+2}-2s+\varepsilon}
	\|W_1\|_{L_x^4 L_t^2}
	\|W_2\|_{L_x^4 L_t^2}.
	$$
\end{proofof}

\begin{proof}[\textbf{Proof of \Cref{t:maximal ose T,t:maximal ose wave}}]
	Since $\alpha'\le2$ we have
$\ell^{\alpha'}\subseteq\ell^2$ with
$\|\lambda\|_{\ell^2}\le\|\lambda\|_{\ell^{\alpha'}}$, so it suffices to prove the
case $\alpha'=2$. The estimates \eqref{e:max ose T } and \eqref{e:max ose wave} are equivalent to
	\[
	\Big\|
	\sum_j \lambda_j \bigl|e^{it\Delta}\langle\nabla\rangle^{-s} f_j\bigr|^2
	\Big\|_{L_x^2L_t^\infty(X \times I)}
	\le C\|\lambda\|_{\ell^2},
	\]
	for any orthonormal system $(f_j)_j$ in $L^2(\M),$
	where $X \times I$ is as in \eqref{XxI}. Let us define 
	\begin{equation} \label{sigma for T and wave}
		\sigma= 
		\frac{d}{2}+\frac{d}{d+2} .
	\end{equation}
	Take the dyadic family $\{\psi_N\}_{N\in2^{\mathbb N}}$ defined in \eqref{psi N}. Using the facts that
	\[
	\sum_{N\in2^{\mathbb N}}\psi_N=1,
	\qquad
	\operatorname{supp}\psi_N=\{\xi:\ |\xi|\sim N\},
	\]
	and applying the triangle inequality together with Minkowski’s inequality, we obtain
	\[
	\Big\|
	\sum_j \lambda_j \bigl|e^{it\Delta}\langle\nabla\rangle^{-s} f_j\bigr|^2
	\Big\|_{L_x^2L_t^\infty(X \times I)}
	\le
	\Biggl(
	\sum_{N\in2^{\mathbb N}}
	\Big\|
	\sum_j \lambda_j \bigl|e^{it\Delta}\langle\nabla\rangle^{-s} P_N f_j\bigr|^2
	\Big\|_{L_x^2L_t^\infty(X \times I)}^{1/2}
	\Biggr)^2.
	\]
	Now by \Cref{prop:strichartz for each PN T,prop:strichartz for each PN wave}, we obtain
	\[
	\Big\|
	\sum_j \lambda_j \bigl|e^{it\Delta}\langle\nabla\rangle^{-s} f_j\bigr|^2
	\Big\|_{L_x^2L_t^\infty(X \times I)}
	\lesssim_\varepsilon
	\Biggl(
	\sum_{N\in2^{\mathbb N}}
	N^{\frac{1}{2}(\sigma-2s+\varepsilon)}
	\Biggr)^2 \|\lambda\|_{\ell^2}.
	\]
	If  $\sigma-2s <0$, or equivalently $s>\frac{\sigma}{2}$, choose $\varepsilon>0$ so that
$\sigma-2s+\varepsilon<0$. Then the dyadic series converges:
	\[
	\sum_{N\in2^{\mathbb N}} N^{\frac{1}{2}(\sigma-2s+\varepsilon)}<\infty .
	\]
	This completes the proof.
\end{proof}
\subsection{Proof of \Cref{c:pointwise T,c:pointwise wave}}\label{sec7}
We begin by recalling the Sobolev-Schatten class introduced in \Cref{schp}. For $s\ge0$, let
\[ \Sp^{\alpha',s}:=\left\{\gamma\in\Com(H^{-s}(\M),H^{s}(\M)):\|\langle\nabla\rangle^{s}\gamma\langle\nabla\rangle^{s}\|_{\Sp^{\alpha'}(L^2(\M))}<\infty \right\},\]
where $\Com(H^{-s}(\M),H^{s}(\M))$ denotes the space of compact operators from $H^{-s}(\M)$ to $H^{s}(\M)$.
If $\gamma_0\in\Sp^{\alpha',s}$ is self-adjoint, then by the finite-rank approximation, there exists a sequence $(\lambda_j)_j\in\ell^{\alpha'}$ and an orthonormal system $(f_j)_j$ in $L^2(\M)$ such that
\[
\gamma_0^N
=
\sum_{j=1}^N
\lambda_j
\langle\nabla\rangle^{-s}
|f_j\rangle\langle f_j|
\langle\nabla\rangle^{-s},
\]
and
\[
\lim_{N\to\infty}
\|
\langle\nabla\rangle^{s}
(\gamma_0^N-\gamma_0)
\langle\nabla\rangle^{s}
\|_{\Sp^{\alpha'}(L^2(\M))}
=0.
\]

The density function associated with $\gamma_0^N$ is defined by
\[
\rho_{\gamma_0^N}(z)
=
\sum_{j=1}^{N}
\lambda_j
|\langle\nabla\rangle^{-s}f_j(z)|^2.
\]

Let
\[
\gamma(t)=e^{it\Delta}\gamma_0e^{-it\Delta}
\]
be the solution of \eqref{Hartree}. Recall \eqref{XxI} \begin{equation*}
	X =\begin{cases}
		\T^d  \quad &if \quad \M=\T^d \\
		B^n(0,1) \times \T^m  \quad &if \quad  \M=\R^n \times \T^m
	\end{cases}.
\end{equation*}
We next explain how to define the density functions $\rho_{\gamma_0}$ and $\rho_{\gamma(t)}$. To this end, we first prove that for $\frac{d}{4}\le s<\frac{d}{2}$,
\begin{equation}\label{e:lieb-sobolev}
	\bigg\|
	\sum_j
	\lambda_j
	|\langle\nabla\rangle^{-s}f_j|^2
	\bigg\|_{L^2(X)}
	\lesssim
	\|\lambda\|_{\ell^{\alpha'}}
\end{equation}
holds for every $\alpha'<2$, every $(\lambda_j)_j\in\ell^{\alpha'}$, and every orthonormal system $(f_j)_j$ in $L^2(\M)$.
By \Cref{lem:lieb-sobolev-product} with $p=2$ gives LHS of \eqref{e:lieb-sobolev},
\begin{equation}\label{e:Lieb_Sobolev}
	\bigg\| \sum_j \lambda_j |\langle \nabla \rangle^{-s} f_j|^2 \bigg\|_{L^2(X)}  \lesssim \| \lambda \|_{\ell^{2,1}}.
\end{equation}

In view of the inclusion\footnote{
	If $\alpha' < 2$ and $(\lambda_j^*)_j$ is the sequence $(|\lambda_j|)_j$ permuted in a decreasing order, we have
	$
	\| \lambda \|_{\ell^\alpha} \lesssim (\sum_{j \geq 1} (\lambda_j^*)^{\alpha} j^{\alpha/2} \cdot j^{-\alpha/2})^{1/\alpha} \lesssim  \sup_{j \geq 1} j^{1/2} \lambda_j^* = \|\lambda\|_{\ell^{2,\infty}}
	$
	and therefore, by duality, $\ell^{\alpha'} \subseteq \ell^{2,1}$. } $\ell^{\alpha'} \subseteq \ell^{2,1},$ we get the desired estimate \eqref{e:lieb-sobolev}.

We first note that in the finite-rank case; that is, for each fixed $N \in \mathbb{N},$ we have $\lim_{t \to 0}\rho_{\gamma^{N}(t)}(z)=\rho_{\gamma_{0}^{N}}(z)$ for a.e $z \in X$ holds by \Cref{ptonT,genind}.

We now fix $\alpha'<2$ and approximate $\gamma_0 = \sum_{j=1}^\infty \lambda_j \langle \nabla \rangle^{-s}\lvert f_j \rangle \langle f_j \rvert \langle \nabla \rangle^{-s} \in \Sp^{\alpha',s}$, for $\lambda \in \ell^{\alpha'}$ and orthonormal vectors $f_j \in L^2(\M)$, by the sequence of finite-rank operators $(\gamma_0^N)_{N\ge1}$ given by $\gamma_0^N = \sum_{j=1}^N \lambda_j \langle \nabla \rangle^{-s}\lvert f_j \rangle \langle f_j \rvert \langle \nabla \rangle^{-s}$. For $M > N$, we obtain
\[
\| \rho_{\gamma_0^N} - \rho_{\gamma_0^M} \|_{L^2(X)} = \bigg\| \sum_{j=N+1}^M \lambda_j |\langle \nabla \rangle^{-s}f_j|^2 \bigg\|_{L^2(X)} \lesssim \bigg( \sum_{j=N+1}^M |\lambda_j|^{\alpha'} \bigg)^{\frac1{\alpha'}}
\]
from \eqref{e:lieb-sobolev}, and therefore $(\rho_{\gamma_0^N})$ is a Cauchy sequence in $L^2(X)$.
Thus, we define $\rho_{\gamma_0} = \sum_{j=1}^\infty \lambda_j |\langle \nabla \rangle^{-s}f_j|^2 \in L^2(X)$ as the limit of $(\rho_{\gamma_0^N})$ in $L^2(X)$. Since orthonormality of $(f_j)_j$ is preserved under the action of $e^{it\Delta}$ for each $t \in \mathbb{R}$, we may repeat the above to define the density function $\rho_{\gamma(t)} = \sum_{j=1}^\infty \lambda_j|\langle \nabla \rangle^{-s}e^{it\Delta} f_j |^2 \in L^{2}(X)$. \\
Let $\alpha'<2,$ $\frac{d}{4}+\frac{d}{2(d+2)}<s<\frac{d}{2}$ and $\gamma_0\in \Sp^{\alpha',s}$, and let $\gamma(t) = e^{it\Delta} \gamma_0 e^{-it\Delta}$. \\
To prove Corollaries \ref{c:pointwise T} and \ref{c:pointwise wave}, it suffices to prove
\begin{equation} \label{e:pointwisegoal}
	|\{z \in X: \limsup_{t \to 0}{|\rho_{\gamma(t)}(z)-\rho_{\gamma_{0}}(z)|} >k^{-1}\}|=0
\end{equation}
for each integer $k \geq 1,$ and $|U|$ is the Lebesgue measure of the set $U.$ And to see this, we fix $\varepsilon>0$ and note
\begin{align*}
	&\left|\left\{z\in X:\limsup_{t\to 0}
	\left|\rho_{\gamma(t)}(z)-\rho_{\gamma_0}(z)\right|>k^{-1}\right\}\right| \\
	&\leq 
	\left|\left\{z\in X:
	\sup_{t\in(0,1]}
	\left|\rho_{\gamma(t)}(z)-\rho_{\gamma_N(t)}(z)\right|
	>(3k)^{-1}\right\}\right| \\
	&\quad +
	\left|\left\{z\in X:
	\limsup_{t\to 0}
	\left|\rho_{\gamma^N(t)}(z)-\rho_{\gamma_0^N}(z)\right|
	>(3k)^{-1}\right\}\right| \\
	&\quad +
	\left|\left\{z\in X:
	\left|\rho_{\gamma_0^N}(z)-\rho_{\gamma_0}(z)\right|
	>(3k)^{-1}\right\}\right| \\
	&=: M_1+M_2+M_3 ,
\end{align*}
where $N$ is to be chosen momentarily. \\
For \(M_3\), by Chebyshev’s inequality and \eqref{e:lieb-sobolev}, we have
\[
M_3 \leq (3k)^2 \|\rho_{\gamma_0^N}-\rho_{\gamma_0}\|_{L^2(X)}^2 < \varepsilon
\]
if we take \(N=N(\varepsilon,k)\) sufficiently large. For \(M_1\), we use Chebyshev’s inequality and \Cref{t:maximal ose T,t:maximal ose wave} to estimate
\[
M_1 \leq (3k)^2 \|\rho_{\gamma(t)}-\rho_{\gamma^N(t)}\|_{L_x^2L_t^\infty(X \times (0,1])}^2
\lesssim
(3k)^2 \left(\sum_{j>N} |\lambda_j|^{\alpha'} \right)^{\frac{2}{\alpha'}}
< \varepsilon,
\]
for \(N=N(\varepsilon,k)\) sufficiently large. For \(M_2\), by the above observation in the finite-rank case, it follows that \(M_2=0\) for any choice of \(N\). Hence, we obtain \eqref{e:pointwisegoal}.

Now in the case of waveguide manifold i.e., when $\M=\R^n \times \T^m, $ we have already shown \eqref{e:pointwisegoal} on $X=B^n(0,1)\times\mathbb T^m.$ The same argument applies to any translated set \(B^n(x_0,1)\times\mathbb T^m\).
Let $x_0\in\mathbb R^n$ and write $x=x_0+x_1,$ where
$x_1\in B^n(0,1)$. Define
$$
\hat g_j(\xi,k)=e^{2\pi i x_0\cdot \xi}\hat f_j(\xi,k),
$$
so $(g_j)_j$ is an orthonormal system if and only if $(f_j)_j$ is an orthonormal system, and note that
$$
e^{it\Delta}f_j(x,y)=e^{it\Delta}g_j(x_1,y).
$$
By \Cref{t:maximal ose wave},
$$
\|\sum_j\lambda_j|e^{it\Delta}f_j|^2\|_{L_x^2L_t^\infty(B^n(x_0,1)\times\mathbb T^m\times[0,1])}
=
\|\sum_j\lambda_j|e^{it\Delta}g_j|^2\|_{L_x^2L_t^\infty(B^n(0,1)\times\mathbb T^m\times[0,1])}
\leq C\|\lambda\|_{\ell^{\alpha'}}.
$$
Therefore,
$$
|\{z\in B^n(x_0,1)\times\mathbb T^m : \limsup_{t\to 0}|\rho_{\gamma(t)}(z)-\rho_{\gamma_0}(z)|>k^{-1}\}|=0.
$$
which finishes the proof.

\section{Necessary condition for infinite system}\label{sec-necc-inf}

\begin{proofof}{necessary for infinite}
	We prove \Cref{necessary for infinite} by contradiction, if possible let $\gamma_0 \in \Sp^{\alpha',s}(L^2(\M))$ be self-adjoint with $\alpha' \geq 1$ and $s < \frac{d}{2(d+1)}.$ Suppose that the densities
	$\rho_{\gamma(t)}(z)$ and $\rho_{\gamma_0}(z)$ are well defined, and they satisfy
	pointwise convergence 
	\begin{equation}\label{density convergence corolary}
		\lim_{t\to0} \rho_{\gamma(t)}(z) = \rho_{\gamma_0}(z)\quad { for \ \rm a.e.}\;\;\; z\in\M.
	\end{equation} 
	Recall that  the Bessel potential operator $\langle\nabla\rangle^{-s}:L^2(\M)\longrightarrow H^s(\M)$,
	which is a self-adjoint isometric isomorphism, with inverse
	$\langle\nabla\rangle^{s}:H^s(\M)\longrightarrow L^2(\M)$. In particular, for
	$f\in L^2(\M)$,
	\[
	\|\langle\nabla\rangle^{-s}f\|_{H^s(\M)}=\|f\|_{L^2(\M)},
	\]
	and a sequence $(f_j)_{j\ge1}\subset L^2(\M)$ is orthonormal in $L^2(\M)$ if
	and only if $(\langle\nabla\rangle^{-s}f_j)_{j\ge1}$ is orthonormal in $H^s(\M)$.
	
	Let $F\in H^s(\M)$ be a function for which
	\begin{equation}\label{existence of non convergence function}
		\lim_{t\to0} e^{it\Delta}F(z) \neq F(z) \qquad \text{for a.e. } z\in\M,
	\end{equation}
	whose existence is guaranteed by \Cref{Necessary condition for theorem} and \Cref{Necessary condition for torus}.
	We will divide the proof into three steps:
	\begin{itemize}
		\item [--] \textbf{Packing $F$ into an operator}
	\end{itemize}
	Set $f_1:=\langle\nabla\rangle^s F\in L^2(\M)$, so that
	$\|f_1\|_{L^2(\M)}=\|F\|_{H^s(\M)}$, which we normalize to equal $1$. We
	extend $f_1$ to an $L^2(\M)$-orthonormal sequence $(f_j)_{j\ge1}\subset
	L^2(\M)$ by the Gram--Schmidt orthonormalization procedure, and set
	\[
	g_j:=\langle\nabla\rangle^{-s}f_j\in H^s(\M), \qquad j\ge1,
	\]
	so that $(g_j)_{j\ge1}$ is $H^s(\M)$-orthonormal with $g_1=\langle\nabla\rangle^{-s}f_1=F$.
	Choose
	\[
	\lambda_1=1,\qquad \lambda_j=0\ \ (j\ge2).
	\]
	Define
	\[
	\gamma_0:=\sum_{j=1}^{\infty}\lambda_j\,\langle\nabla\rangle^{-s}|f_j\rangle\langle f_j|\langle\nabla\rangle^{-s}
	=\langle\nabla\rangle^{-s}|f_1\rangle\langle f_1|\langle\nabla\rangle^{-s}.
	\]
	Then $\gamma_0$ is self-adjoint, and $\langle\nabla\rangle^s\gamma_0\langle\nabla\rangle^s=|f_1\rangle\langle f_1|$
	is a rank-one operator, hence lies in $\Sp^{\alpha'}(L^2(\M))$ for every
	admissible $\alpha'$; thus $\gamma_0\in\Sp^{\alpha',s}$. Its density is
	\[
	\rho_{\gamma_0}(z)=|g_1(z)|^2=|F(z)|^2, \qquad z\in\M.
	\]
	
	Under the free Schr\"odinger evolution, since $e^{it\Delta}$ and
	$\langle\nabla\rangle^{-s}$ are both Fourier multipliers and therefore
	commute,
	\[
	\gamma(t):=e^{it\Delta}\gamma_0e^{-it\Delta}
	=\langle\nabla\rangle^{-s}|e^{it\Delta}f_1\rangle\langle e^{it\Delta}f_1|\langle\nabla\rangle^{-s},
	\]
	so that its density is
	\[
	\rho_{\gamma(t)}(z)=\big|\langle\nabla\rangle^{-s}(e^{it\Delta}f_1)(z)\big|^2
	=\big|e^{it\Delta}\langle\nabla\rangle^{-s}f_1(z)\big|^2
	=|e^{it\Delta}F(z)|^2, \qquad z\in\M.
	\]
	Hence, by the density convergence hypothesis \eqref{density convergence corolary},
	\begin{equation}\label{square convergence}
		\lim_{t\to0}|e^{it\Delta}F(z)|^2 = |F(z)|^2 \quad \text{a.e. } z\in\M.
	\end{equation}
	The convergence stated above \eqref{square convergence} is true for all functions \( f \in H^s(\M), \) because we can apply similar arguments for each non-zero function, while the case for the zero function is trivial. However, the convergence of the modulus alone does not guarantee the pointwise convergence of \( e^{it\Delta}F \) itself. Therefore, we will use some identities from complex analysis to establish this pointwise convergence.
	
	\begin{itemize}
		\item [--] \textbf{Getting the phase back}
	\end{itemize}
	Fix a nowhere zero Schwartz function $h\in C^\infty(\M)\cap H^s(\M)$ such
	that
	\[
	\lim_{t\to0} e^{it\Delta}h(z) = h(z) \qquad \text{for every } z\in\M.
	\]
	In particular, if $z=(x,y)\in\M=\R^n\times\T^m$, we may take
	\[
	h(z)=e^{-|x|^2},
	\]
	which is smooth, lies in $H^s(\M)$ for every $s$ (Schwartz in $x$, constant
	in $y$), is strictly positive, and whose Schr\"odinger evolution converges
	pointwise everywhere, since $h$ is Schwartz in the $\R^n$ variable.
	
	For a complex number $z_1\in\{1,-1,i,-i\}$, by \eqref{square convergence},
	\[
	\lim_{t\to0}|e^{it\Delta}F(z)+z_1e^{it\Delta}h(z)|^2 = |F(z)+z_1h(z)|^2 \quad \text{a.e. } z\in\M.
	\]
	
	Taking $z_1=1$ and $z_1=-1$ gives, respectively,
	\begin{equation}\label{eq:z1=1}
		\lim_{t\to0}\big|e^{it\Delta}F(z)+e^{it\Delta}h(z)\big|^2 = |F(z)+h(z)|^2 \quad \text{a.e. } z\in\M,
	\end{equation}
	\begin{equation}\label{eq:z1=-1}
		\lim_{t\to0}\big|e^{it\Delta}F(z)-e^{it\Delta}h(z)\big|^2 = |F(z)-h(z)|^2 \quad \text{a.e. } z\in\M.
	\end{equation}
	Subtracting \eqref{eq:z1=-1} from \eqref{eq:z1=1} and using the identity
	$|u+v|^2-|u-v|^2=4\operatorname{Re}(u\overline v)$,
	\begin{equation}\label{eq:real-part}
		\lim_{t\to0}\operatorname{Re}\big(e^{it\Delta}F(z)\,\overline{e^{it\Delta}h(z)}\big)
		= \operatorname{Re}\big(F(z)\overline{h(z)}\big) \quad \text{a.e. } z\in\M.
	\end{equation}
	
	Taking $z_1=i$ and $z_1=-i$ gives, respectively,
	\begin{equation}\label{eq:z1=i}
		\lim_{t\to0}\big|e^{it\Delta}F(z)+ie^{it\Delta}h(z)\big|^2 = |F(z)+ih(z)|^2 \quad \text{a.e. } z\in\M,
	\end{equation}
	\begin{equation}\label{eq:z1=-i}
		\lim_{t\to0}\big|e^{it\Delta}F(z)-ie^{it\Delta}h(z)\big|^2 = |F(z)-ih(z)|^2 \quad \text{a.e. } z\in\M.
	\end{equation}
	Subtracting \eqref{eq:z1=-i} from \eqref{eq:z1=i} and using
	$|u+iv|^2-|u-iv|^2=4\operatorname{Im}(u\overline v)$,
	\begin{equation}\label{eq:imag-part}
		\lim_{t\to0}\operatorname{Im}\big(e^{it\Delta}F(z)\,\overline{e^{it\Delta}h(z)}\big)
		= \operatorname{Im}\big(F(z)\overline{h(z)}\big) \quad \text{a.e. } z\in\M.
	\end{equation}
	Combining \eqref{eq:real-part} and \eqref{eq:imag-part},
	\begin{equation}\label{eq:product-convergence}
		\lim_{t\to0} e^{it\Delta}F(z)\,\overline{e^{it\Delta}h(z)} = F(z)\overline{h(z)} \quad \text{a.e. } z\in\M.
	\end{equation}
	\begin{itemize}
		\item [--] \textbf{Dividing $h$ back out}
	\end{itemize}
	By the choice of $h$, $\lim_{t\to0} e^{it\Delta}h(z)=h(z)$ for every
	$z\in\M$, and $h$ is nowhere zero, so $\overline{e^{it\Delta}h(z)}\neq0$ for
	all sufficiently small $t$, at every point where
	\eqref{eq:product-convergence} holds. Thus, writing
	\[
	e^{it\Delta}F(z) = \frac{e^{it\Delta}F(z)\,\overline{e^{it\Delta}h(z)}}{\overline{e^{it\Delta}h(z)}},
	\]
	and passing to the limit using \eqref{eq:product-convergence} together with
	$\lim_{t\to0}e^{it\Delta}h(z)=h(z)$,
	\[
	\lim_{t\to0} e^{it\Delta}F(z) = \frac{F(z)\overline{h(z)}}{\overline{h(z)}} = F(z) \quad \text{a.e. } z\in\M.
	\]
	
	This contradicts \eqref{existence of non convergence function}, which
	guarantees, for $s<\frac{d}{2(d+1)}$, the existence of $F\in H^s(\M)$
	for which pointwise convergence fails on a set of positive measure.
	Consequently,
	\[
	s\geq\dfrac{d}{2(d+1)}.
	\]    
\end{proofof}

\section*{Acknowledgments}
S.~R.~C. thanks the National Board for Higher Mathematics (NBHM) for a
research fellowship, and IISER Pune for its support. D.~G.~B. is grateful to an ANRF-MATRICS Grant ANRF/ARGM/2025/000580/MTR.

\printbibliography

\end{document}